\documentclass[english]{article}
\usepackage[T1]{fontenc}
\usepackage[latin9]{inputenc}
\usepackage{geometry}
\usepackage{array}
\usepackage{mathtools}
\usepackage{multirow}
\usepackage{amsmath}
\usepackage{amsthm}
\usepackage{amssymb}
\usepackage{graphicx}
\usepackage{enumitem}
\usepackage{authblk}

\makeatletter

\providecommand{\tabularnewline}{\\}

\numberwithin{equation}{section}
\numberwithin{figure}{section}
\theoremstyle{plain}
\newtheorem{thm}{\protect\theoremname}
\theoremstyle{remark}
\newtheorem{rem}[thm]{\protect\remarkname}
\theoremstyle{definition}
\newtheorem{defn}[thm]{\protect\definitionname}
\theoremstyle{plain}
\newtheorem{lem}[thm]{\protect\lemmaname}

\usepackage{color}
\usepackage{hyperref}

\newcommand{\REV}[1]{#1}

\allowdisplaybreaks

\makeatother

\usepackage{babel}
\providecommand{\definitionname}{Definition}
\providecommand{\lemmaname}{Lemma}
\providecommand{\remarkname}{Remark}
\providecommand{\theoremname}{Theorem}

\begin{document}
\global\long\def\ve{\varepsilon}%
\global\long\def\R{\mathbb{R}}%
\global\long\def\Rn{\mathbb{R}^{n}}%
\global\long\def\Rd{\mathbb{R}^{d}}%
\global\long\def\E{\mathbb{E}}%
\global\long\def\P{\mathbb{P}}%
\global\long\def\bx{\mathbf{x}}%
\global\long\def\vp{\varphi}%
\global\long\def\ra{\rightarrow}%
\global\long\def\smooth{C^{\infty}}%
\global\long\def\Tr{\mathrm{Tr}}%
\global\long\def\bra#1{\left\langle #1\right|}%
\global\long\def\ket#1{\left|#1\right\rangle }%
\global\long\def\Re{\mathrm{Re}}%
\global\long\def\Im{\mathrm{Im}}%
\global\long\def\bsig{\boldsymbol{\sigma}}%
\global\long\def\btau{\boldsymbol{\tau}}%
\global\long\def\bmu{\boldsymbol{\mu}}%
\global\long\def\bx{\boldsymbol{x}}%
\global\long\def\bups{\boldsymbol{\upsilon}}%
\global\long\def\bSig{\boldsymbol{\Sigma}}%
\global\long\def\bt{\boldsymbol{t}}%
\global\long\def\bs{\boldsymbol{s}}%
\global\long\def\by{\boldsymbol{y}}%
\global\long\def\brho{\boldsymbol{\rho}}%
\global\long\def\ba{\boldsymbol{a}}%
\global\long\def\bb{\boldsymbol{b}}%
\global\long\def\bz{\boldsymbol{z}}%
\global\long\def\bc{\boldsymbol{c}}%
\global\long\def\balpha{\boldsymbol{\alpha}}%
\global\long\def\bbeta{\boldsymbol{\beta}}%
\global\long\def\T{\mathrm{T}}%
\global\long\def\trip{\vert\!\vert\!\vert}%
\global\long\def\lrtrip#1{\left|\!\left|\!\left|#1\right|\!\right|\!\right|}%

\author[1]{P. Michael Kielstra}
\author[1,2]{Michael Lindsey}
\affil[1]{University of California, Berkeley}
\affil[2]{Lawrence Berkeley National Laboratory}
\renewcommand\Affilfont{\itshape\small}

\title{Gaussian process regression with log-linear scaling for common non-stationary kernels }
\maketitle
\begin{abstract}
We introduce a fast algorithm for Gaussian process regression in low
dimensions, applicable to a widely-used family of non-stationary kernels.
The non-stationarity of these kernels is induced by arbitrary spatially-varying
vertical and horizontal scales. In particular, any stationary kernel
can be accommodated as a special case, and we focus especially on
the generalization of the standard Mat\'ern kernel. Our subroutine for
kernel matrix-vector multiplications scales almost optimally as $O(N\log N)$,
where $N$ is the number of regression points. Like the recently developed
equispaced Fourier Gaussian process (EFGP) methodology, which is applicable
only to stationary kernels, our approach exploits non-uniform fast
Fourier transforms (NUFFTs). We offer a complete analysis controlling
the approximation error of our method, and we validate the method's
practical performance with numerical experiments. In particular we
demonstrate improved scalability compared to to state-of-the-art rank-structured
approaches in spatial dimension $d>1$.
\end{abstract}

\section{Introduction\label{sec:Introduction}}

Gaussian Process Regression (GPR) is widely used across many scientific
fields \cite{gpex1,gpex2,foreman-gpex3,gpex4,gpex5,rasmussen_gaussian_2006}
as a framework for inferring a function $f:\R^{d}\ra\R$ from noisy
observations at an arbitrary collection of scattered points $x_{1},\ldots,x_{N}\in\R^{N}$.
The success of GPR owes to its convenient linear-algebraic algorithmic
formulation and its capacity for interpretable uncertainty quantification.

GPR is based on the selection of a prior distribution over functions,
induced by a choice of positive definite kernel \cite{rasmussen_gaussian_2006}
$\mathcal{K}(x,y)$ defined for $x,y\in\R^{d}$. The key operations
of GPR can be phrased linear-algebraically in terms of the kernel
matrix $\mathbf{K}=(K(x_{i},x_{j}))_{i,j=1}^{N}$. In particular,
assuming a noise model in which the observations $y_{n}=f(x_{n})+\epsilon_{n}$
are perturbed from the true values by independent and identically
distributed Gaussian noise terms $\epsilon_{n}\sim\mathcal{N}(0,\eta^{2})$,
the key algorithmic step in GPR is the solution of the linear system
\begin{equation}
\left(\mathbf{K}+\eta^{2}I_{N}\right)\alpha=y,\label{eq:gprsolve}
\end{equation}
 where $y=(y_{1},\ldots,y_{N})^{\top}$ is the vector of observations.
In turn the mean of the posterior distribution over functions, given
the observations, can be recovered \cite{rasmussen_gaussian_2006}
from the solution vector $\alpha$ as 
\begin{equation}
\mu(x)=\sum_{n=1}^{N}\alpha_{n}\,\mathcal{K}(x,x_{n}).\label{eq:postmean}
\end{equation}
 Uncertainty of this estimator can be quantified using the covariance
function of the posterior, which can also be written linear-algebraically
\cite{rasmussen_gaussian_2006} as 
\begin{equation}
\Sigma(x,y)=\mathcal{K}(x,y)-\mathcal{K}(x,x_{m})\left[\left(\mathbf{K}+\eta^{2}I_{N}\right)^{-1}\right]_{mn}\mathcal{K}(x_{n},y).\label{eq:postcov}
\end{equation}

A key task in GPR is therefore is to perform linear solves involving
the matrix $\mathbf{K}+\eta^{2}I$. Naively, forming the matrix requires
$\Omega(N^{2})$ operations, and direct solvers require $\Omega(N^{3})$
operations. Since the matrix is positive definite, the conjugate gradient
(CG) method \cite{trefethen_numerical_1997} can be applied using
the dense matrix to potentially reduce the overall cost to $\Omega(N^{2})$,
assuming a condition number independent of $N$, and parallelism may
be exploited to further improve the practical scaling \cite{wang_pleiss}.

Our focus is on approaches with faster than quadratic scaling, which
requires the exploitation of some sort of structure of $\mathbf{K}$.
In particular, we hope for the nearly-optimal scaling of $\tilde{O}(N)$,
where $\tilde{O}$ indicates the omission of logarithmic factors.
Though we will review a few paradigms for fast GPR, we refer the reader
to \cite{greengard_equispaced_2023} for an excellent summary of the
literature.

The most straightforward paradigms include low-rank factorization
of $\mathbf{K}$ \cite{nystrom_uber_1930,rasmussen_gaussian_2006}
and sparse thresholding of\textbf{ $\mathbf{K}$} \cite{furrer_covariance_2006}.
However, there exist many scenarios in which neither the numerical
rank nor sparseness suffice for tractable computations. When the spatial
dimension $d$ is low, more sophisticated structured rank decompositions
have emerged as a paradigm for the compression of $\mathbf{K}$. These
decompositions can also permit direct algorithms for the required
linear solves. Structured formats used for GPR include the hierarchically
off-diagonal low-rank (HODLR) format, as well as the HSS and HBS formats
\cite{ambikasaran_fast_2016,minden_fast_2017,ho_flam_2020,martinsson_fast_2019}.
Such approaches are extremely effective in $d=1$, but their performance
can noticeably degrade even for $d=2$ as the relevant numerical ranks
can grow with both $d$ and $N$.

Greengard et al. \cite{greengard_equispaced_2023} recently introduced
a method for kernel matrix-vector multiplications (or matvecs) with
$\tilde{O}(N)$ cost that eschews both sparse and low-rank compression
entirely. Their equispaced Fourier Gaussian process (EFGP) approach
relies on non-uniform fast Fourier transforms (NUFFTs) \cite{BOYD1992243,DuttRokhlin}.
A detailed analysis controlling the error of EFGP is explored in \cite{barnett_uniform_2023}.
However, a major limitation of this approach is the assumption that
the kernel is \emph{stationary}, i.e., satisfies $\mathcal{K}(x,y)=\mathcal{K}(x-y)$.

In our own language, the idea of \cite{greengard_equispaced_2023}
can be summarized as follows. Suppose that we wish to compute an arbitrary
matvec $\mathbf{K}\alpha$. Observe the identity: 
\[
\left[\mathbf{K}\alpha\right]_{m}=\left[\mathcal{K}\sum_{n=1}^{N}\alpha_{n}\,\delta_{x_{n}}\right](x_{m}),
\]
 where we view $\mathcal{K}$ as an integral operator with kernel
$\mathcal{K}(x,y)$ and let $\delta_{x}$ denote the Dirac delta distribution
localized at $x$. Then letting $\mathcal{F}$ denote the unitary
Fourier transform (and $\mathcal{F}^{*}$ its inverse), we have equivalently
that
\begin{equation}
\left[\mathbf{K}\alpha\right]_{m}=\left[\mathcal{F}^{*}\hat{\mathcal{K}}f\right](x_{m}),\ \,\text{where }\ f:=\mathcal{F}\sum_{n=1}^{N}\alpha_{n}\delta_{x_{n}},\label{eq:efgp}
\end{equation}
 and $\hat{\mathcal{K}}:=\mathcal{F}\hat{\mathcal{K}}\mathcal{F}^{*}$.
Since $\mathcal{K}$ is stationary, it follows that $\hat{\mathcal{K}}$
is a diagonal operator. Meanwhile, $f$ as defined in (\ref{eq:efgp})
can be computed on a grid in Fourier space using a type 1 NUFFT (cf.
Appendix \ref{app:nuffts} for further background). Once $\hat{\mathcal{K}}f$
is formed on the Fourier grid, then the entries of $\mathbf{K}\alpha$
can be recovered via (\ref{eq:efgp}) using a type 2 NUFFT (again
refer to Appendix \ref{app:nuffts}). The total cost is $O(M^{d}+N\log N)$,
where $M$ is the number of discretization points per dimension in
Fourier space. Viewing $M$ as a constant controlling the error of
the method, the matvec scaling is $O(N\log N)$ as desired.

Our work aims to extend the success of EFGP to the commonly used family
of \emph{non-stationary} \emph{kernels} (especially, non-stationary,
\emph{Mat\'ern }kernels), developed in \cite{paciorek_nonstationary_2003,paciorek_c_nonstationary_2004}
and recently applied in contemporary settings requiring flexible kernel
design \cite{NoackEtAl2023, NoackEtAl2024}. \REV{Settings in which the non-stationary Mat\'ern kernel has been shown to be well-adapted include IR and neutron spectroscopy \cite{Noack_Sethian_2022} and climatology \cite{Paciorek_Schervish_2006}. Even without non-stationary capabilities, the classical Mat\'ern kernel has been successfully applied in fields ranging from astronomy and astrophysics \cite{Griffiths_2021, Benisty_2021} to nanofluid modeling \cite{Mukesh_2021} and chemical engineering \cite{Pustokhina_Seraj_Hafsan_Mostafavi_Alizadeh_2021}.}

Our strategy is made up of several intermediate steps that allow us to reduce,
like \cite{greengard_equispaced_2023}, to diagonal operations on
a Fourier grid, though the operations that we require are different.
First, we exploit the `Schoenberg' representation (cf. \cite[Theorem 2]{schoenberg_metric_1938-1}
as well as (\ref{eq:schoenberg}) below) of the positive definite
function that induces the non-stationary kernel \cite{paciorek_c_nonstationary_2004},
to essentially reduce to the case of non-stationary squared exponential
kernels. In fact, this Schoenberg representation is fundamental to the
derivation \cite{paciorek_c_nonstationary_2004} of the non-stationary
kernel family of interest. Second, we use interpolation by value of
the non-stationary scale function $\sigma(x)$ that defines our non-stationary
kernel (cf. (\ref{eq:isotrop}) below) to reduce the application of
the kernel to the task of applying several Gaussian convolutions,
which can be achieved as diagonal operations in Fourier space. Like
\cite{greengard_equispaced_2023}, we require NUFFTs of both type
1 and type 2 to pass from the scattered grid to the equispaced Fourier
grid and vice versa.

\REV{We prove exponential convergence of our approximation with respect to all hyperparameters of the algorithm that control the error. In terms of the target error $\epsilon$ as well as the Mat\'ern parameter $\nu$, cf. \eqref{eq:matern}, the computational complexity of the approach is $O(\epsilon^{-d p})$ where $d$ is the spatial dimension and $p$ is any exponent larger than $(2 \nu)^{-1}$. The scaling is due to the necessity of constructing a Fourier grid with $\epsilon^{-p}$ points per dimension.
For a rigorous statement and further detail, see Theorem \ref{thm:main} and the discussion that immediately follows it. Interestingly,
even in the special case of a stationary kernel, our analysis differs
from that of \cite{greengard_equispaced_2023} due to the substitution
of the Schoenberg representation of the kernel, which we discretize
(cf. Section \ref{sec:approxt} and Lemma \ref{lem:dt} below) with
rapidly-converging numerical quadrature. Although passing through the Schoenberg representation is apparently
a forced move in our analysis, due to the structure of our non-stationary
kernel family, ultimately our error analysis yields almost the same complexity as that of 
\cite{greengard_equispaced_2023}
for stationary Mat\'ern kernels, where a Fourier grid with $O(\epsilon^{-1/(2\nu)})$ points per dimension is likewise constructed.}

We demonstrate the effectiveness of our method with several numerical
experiments, and in particular we compare our results to the highly
performant FLAM package \cite{ho_flam_2020} which is based on hierarchical
low-rank decompositions.

As in \cite{greengard_equispaced_2023}, we do not discuss preconditioning.
Potentially, our method for kernel matvecs could be combined with
a preconditioner based on a low-accuracy hierarchical decomposition.
We highlight this direction, as well as other approaches for preconditioning,
as interesting topics for future research.

\subsection{Outline}

In Section \ref{sec:prelim} we present the family of non-stationary
kernels that are of interest in this work. In Section \ref{sec:kernelapprox}
we discuss our approximation framework for this kernel, which is based
on (1) quadrature for the Schoenberg representation, (2) interpolation
by value of the scale function, and (3) Fourier space discretization.
This framework motivates the explicit algorithm presented in Section
\ref{sec:algo}. We outline our error analysis in Section \ref{sec:error}
and present relevant numerical experiments and benchmarks in Section
\ref{sec:numerics}.

We provide a glossary of some of our key notation in Appendix \ref{app:glossary}
and review some background on NUFFTs in Appendix \ref{app:nuffts}.
The remaining appendices give more details on proofs that we defer
in the main body of the paper: the general integral representation of
the kernel in Appendix \ref{app:pres}, the construction of an explicit
Schoenberg representation for the Mat\'ern kernel specifically in Appendix
Appendix \ref{app:matern}, and the various technical components of our error
bound in Appendix \ref{app:lemmas}. The lemmas are synthesized in
Appendix \ref{app:main}, the proof of our main theorem on the error
bound.

\subsection{Acknowledgments}

This work was partially supported by the Applied Mathematics Program
of the US Department of Energy (DOE) Office of Advanced Scientific
Computing Research under contract number DE-AC02-05CH11231 (M.L.).

\section{Preliminaries \label{sec:prelim}}

A rich family of non-stationary kernels, introduced in \cite{paciorek_nonstationary_2003,paciorek_c_nonstationary_2004},
is specified by the general functional form 
\begin{equation}
\mathcal{K}(x,y)=w(x)\,w(y)\,\left|\Sigma(x)+\Sigma(y)\right|^{-1/2}\,\vp\left(\sqrt{(x-y)^{\top}\left[\Sigma(x)+\Sigma(y)\right]^{-1}(x-y)}\right),\label{eq:nonstat}
\end{equation}
 where $\Sigma(x)\succ0$ is an arbitrary $d\times d$ positive-definite-matrix-valued
function of the spatial variable $x\in\R^{d}$, $w(x)\geq0$ is an
arbitrary nonnegative-valued function, vertical bars indicate the
matrix determinant, and $\vp:[0,\infty)\ra\R$ is an arbitrary \emph{radial
positive definite} function, meaning that 
\[
\Phi(x)=\vp(\Vert x\Vert)
\]
 is a positive definite function in the sense of \cite{schoenberg_posdef}.

It is known by a theorem of Schoenberg \cite[Theorem 2]{schoenberg_metric_1938-1}
that $\vp$ is radial positive definite if and only if it can be written
as 
\begin{equation}
\vp(r)=\int_{0}^{\infty}e^{-r^{2}s^{2}}\,d\mu(s),\label{eq:schoenberg}
\end{equation}
 where $\mu$ is a finite nonnegative Borel measure on $[0,\infty)$.
This characterization motivates the intuitive understanding of the
family (\ref{eq:nonstat}) as being generated by nonnegative linear
combinations of `non-stationary squared exponential kernels' of the
form
\[
\mathcal{K}(x,y)=w(x)\,w(y)\,\left|\Sigma(x)+\Sigma(y)\right|^{-1/2}e^{-\frac{1}{2}(x-y)^{\top}\left[\Sigma(x)+\Sigma(y)\right]^{-1}(x-y)}.
\]
 The Gaussian processes induced by such kernels are themselves derived
\cite{paciorek_c_nonstationary_2004} by smoothing a white noise process
with a Gaussian windowing function with spatially dependent covariance.  \REV{More generally, representing a function as an integral over a family of Gaussians, to be approximated with a discrete sum, is a common technique; see, for example, \cite{beylkin_approximation_2005} and \cite{beylkin_approximation_2010}.}

In this work, we focus on the important special case of \emph{non-stationary
isotropic kernels}, in which the structure of $\Sigma(x)$ simplifies
as 
\begin{equation}
\Sigma(x)=\sigma^{2}(x)\,I_{d},\label{eq:isotrop}
\end{equation}
 where $\sigma(x)>0$ is scalar-valued. By a linear change of the
spatial variable, it is easy to reduce to this case from the more
general case 
\[
\Sigma(x)=\sigma^{2}(x)\,\Sigma_{0},
\]
 where $\Sigma_{0}$ is an arbitrary $d\times d$ positive semidefinite
matrix, independent of $x$.

With regard to the selection of $\vp$, of particular interest is
the case of non-stationary Mat\'ern kernels \cite{paciorek_c_nonstationary_2004},
induced by the choice $\vp=\vp_{\nu}$: 
\begin{equation}
\varphi_{\nu}(r)=\frac{1}{2^{\nu-1}\Gamma(\nu)}\left|\sqrt{2\nu}\,r\right|^{\nu}K_{\nu}\left(\sqrt{2\nu}\,r\right),\label{eq:matern}
\end{equation}
where $\nu>0$ is a fixed parameter that governs the smoothness of
kernel and $K_{\nu}$ is the modified Bessel function of the second
kind. (Everywhere, the letter $K$ is reserved for kernel-related
quantities; we use it for the Bessel function only here.)

In general, a function drawn from a GP specified by a Mat\'ern kernel
with $\nu>0$ will be continuous and $\lceil\nu\rceil-1$ times differentiable
\cite[Section 4.2]{rasmussen_gaussian_2006}. In practice, the typically
chosen values of $\nu$ are the half-integers $\frac{1}{2},\,\frac{3}{2},\,\frac{5}{2}\dots$,
since for these values, the Bessel function $K_{\nu}$ admits a closed-form
expression that can be easily evaluted. Since we deal with $\vp_{\nu}$
only via its Schoenberg representation (\ref{eq:schoenberg}), which
we explicitly construct in Appendix \ref{app:pres}, there is no reason
to favor half-integer $\nu$ in our implementation, though the cost
of our algorithm will increase without bound in the $\nu\ra0$ limit.
We remark that $\nu=\frac{1}{2}$ is the lowest value that appears
in common practice, corresponding to the case $\vp(r)\propto e^{-r}$
of the exponential kernel. Meanwhile, in the $\nu\ra\infty$ limit, we recover \cite[Section 4.2]{rasmussen_gaussian_2006}
the squared exponential kernel function 
\begin{equation}
\vp(r)=e^{-\frac{r^{2}}{2}},\label{eq:squaredexp}
\end{equation}
 which is of special interest.

\section{Approximation of the kernel\label{sec:kernelapprox}}

A general non-stationary isotropic kernel in the sense of Section
\ref{sec:prelim} (cf. (\ref{eq:isotrop})) can be written 
\begin{equation}
\mathcal{K}(x,y)=w(x)\,w(y)\,\left(2\pi\left[\sigma^{2}(x)+\sigma^{2}(y)\right]\right)^{-d/2}\,\vp\left(\frac{\vert x-y\vert}{\sqrt{\sigma^{2}(x)+\sigma^{2}(y)}}\right),\label{eq:nonstat2}
\end{equation}
 where $\sigma(x)>0$ and $w(x)\geq0$ are arbitrary. We will assume
upper and lower bounds $\sigma_{\max}$ and $\sigma_{\min}>0$ such
that 
\[
\sigma(x)\in[\sigma_{\min},\sigma_{\max}]
\]
 for all $x\in\R^{d}$. As we shall see, the ratio 
\begin{equation}
\kappa:=\frac{\sigma_{\max}}{\sigma_{\min}}\label{eq:kappa}
\end{equation}
 will in part control the numerical difficulty of representing such
a non-stationary isotropic kernel.

Note that, without any loss of generality, we have absorbed a factor
into the weight function $w(x)$ in order to ease certain manipulations
downstream in the discussion.

\REV{
In order to approximate the integral kernel $\mathcal{K}$ in a framework that will permit fast kernel matrix-vector multiplications, we will adopt the following strategy. First we represent the function $\varphi(r)$ as an integral over a family of Gaussians parametrized by a variable $t$. Then, we discretize the integral, which essentially reduces the treatment of $\mathcal{K}$ to the case where $\vp$ is Gaussian in  \eqref{eq:nonstat2}. In order to remove the non-stationarity in \eqref{eq:nonstat2} due to the spatial dependence of $\sigma (x)$, we approximate $\sigma(x)$ with Chebyshev interpolation. This step allows us to approximate any Gaussian integral kernel with spatially-varying standard deviation as a weighted sum of stationary Gaussian convolutions. To apply the integral operator $\mathcal{K}$, these reductions leave us with a batch of what are now explicitly Gaussian convolutions, which we can compute cheaply using NUFFTs.}

\subsection{Analytical integration in $t$}

Let us suppose for simplicity that $\vp$ is induced in (\ref{eq:schoenberg})
by an absolutely continuous measure $\mu$, as is the case in all
important applications that we shall highlight below. Then it is possible
to write 
\begin{equation}
\vp(r)=\int e^{-\frac{r^{2}}{2\chi^{2}(t)}}\,u(t)\,dt\label{eq:schoenberg2}
\end{equation}
 for suitable functions $\chi:\R\ra(0,\infty)$ and $u:\R\ra[0,\infty)$.
In fact, such a representation is non-unique via change of variables,
but later we shall consider explicit representations that allow for
convenient numerical discretization of the integral $dt$.

In this case, the kernel (\ref{eq:nonstat2}) is recovered exactly
by the representation 
\begin{equation}
\mathcal{K}=\int_{\R}\mathcal{B}_{t}\mathcal{B}_{t}^{*}\,v(t)\,dt,\label{eq:pres}
\end{equation}
 where each $\mathcal{B}_{t}$ (for $t\in\R$) is an integral kernel
defined by 
\begin{equation}
\mathcal{B}_{t}(x,z)=w(x)\left[2\pi\,\sigma^{2}(x)\right]^{-d/2}\,e^{-\frac{\vert x-z\vert^{2}}{2\sigma(x)^{2}\chi(t)^{2}}},\label{eq:Bt}
\end{equation}
 and moreover 
\begin{equation}
v(t):=\chi^{-d}(t)\,u(t),\label{eq:v}
\end{equation}
 as we verify in Appendix \ref{app:pres}. Note that we view $\mathcal{B}_{t}$
as both an integral operator as well as a function $\R^{d}\times\R^{d}\ra\R$
defining the integral kernel of this operator. We will maintain calligraphic
notation for such objects, including $\mathcal{K}$ as well. 

In particular, we compute (cf. (\ref{eq:BtBtstar})) that 
\begin{equation}
\left(\mathcal{B}_{t}\mathcal{B}_{t}^{*}\right)(x,y)\,v(t)=w(x)\,w(y)\left(2\pi\left[\sigma^{2}(x)+\sigma^{2}(y)\right]\right)^{-d/2}\,e^{-\frac{\vert x-y\vert^{2}}{2\left(\sigma(x)^{2}+\sigma(y)^{2}\right)\chi(t)^{2}}}\,u(t).\label{eq:BtBtstar0}
\end{equation}

Meanwhile, as we prove in Appendix \ref{app:matern}, the Mat\'ern function
(\ref{eq:matern}) is recovered exactly by the choice 
\begin{equation}
u(t):=\frac{1}{\Gamma\left(\nu\right)}\,e^{\nu t-e^{t}},\quad\chi(t):=\frac{1}{\sqrt{\nu}}\,e^{t/2}.\label{eq:maternpres}
\end{equation}

In the special case of the squared exponential kernel (\ref{eq:squaredexp}),
there is no real need for integration in $t$. We can formally take
$u(t)=\delta(t)$ and $\chi^{2}(t)\equiv1$ in (\ref{eq:schoenberg2})
to recover this case. 

\subsection{Numerical integration in $t$ \label{sec:approxt}}

We now discuss how to discretize the integral with respect to $t$
in the integral representation (\ref{eq:pres}) of the kernel, in the
Mat\'ern case (\ref{eq:maternpres}). 

In fact, we shall approximate the integral with a simple trapezoidal
Riemann sum:
\begin{equation}
\mathcal{K}\approx\sum_{j=0}^{N_{t}}\mathcal{B}_{t_{j}}\mathcal{B}_{t_{j}}^{*}\,v_{j},\label{eq:trapsum}
\end{equation}
 where 
\[
v_{j}=\begin{cases}
v(t_{j})\,\Delta t, & j\in\{1,\ldots N_{t}-1\},\\
v(t_{j})\,\frac{\Delta t}{2}, & j\in\{0,N_{t}\}.
\end{cases}
\]
 Since the dependence on $t$ of $(\mathcal{B}_{t}\mathcal{B}_{t}^{*})(x,y)\,v(t)$
in (\ref{eq:BtBtstar0}) is analytic with exponentially decaying tails,
we can effectively restrict to a compact domain of integration, and
moreover we expect rapid convergence \cite{trefethen_traprule_2014}
as $\Delta t$ is refined. The error due to this discretization will
be controlled explicitly in Lemma \ref{lem:dt} below.

Specifically, after choosing $t_{\min}<t_{\max}$ bounding the effective
interval of integration and a number of integration points $N_{t}+1$,
we set $\Delta t=\frac{t_{\max}-t_{\min}}{N_{t}}$ and 
\begin{equation}
t_{j}=t_{\min}+\frac{j}{N_{t}}(t_{\max}-t_{\min}),\quad j=0,\ldots,N_{t}.\label{eq:dtscheme}
\end{equation}

In the special case of the squared exponential kernel (\ref{eq:squaredexp}),
there is no need to approximate the integral in $t$. We can recover
this case by taking $N_{t}=0$, $t_{\min}=t_{\max}=0$, and $v_{0}=1$,
as well as $\chi^{2}(0)=1$, as indicated above.

\subsection{Chebyshev interpolation in $\sigma$ \label{sec:approxsigma}}

Each operator $\mathcal{B}_{t}$ can be viewed as the composition
$\mathcal{B}_{t}=\mathcal{D}_{t}\,\mathcal{C}_{t}$ of the diagonal
multiplier 
\begin{equation}
\mathcal{D}_{t}(x,z)=w(x)\left[2\pi\,\sigma^{2}(x)\right]^{-d/2}\delta(x-z)\label{eq:Dt}
\end{equation}
 with the integral operator 
\begin{equation}
\mathcal{C}_{t}(x,z):=e^{-\frac{\vert x-z\vert^{2}}{2\sigma^{2}(x)\chi^{2}(t)}}.\label{eq:Ct}
\end{equation}
From a computational point of view, dealing with $\mathcal{C}_{t}$
is difficult because it is not precisely a convolution operator, due
to the non-stationary dependence $\sigma(x)$. Therefore we are motivated
to replace $\mathcal{C}_{t}$ with a linear combination of operators
which are themselves true convolution operators, and we achieve this
by Chebyshev interpolation with respect to the \emph{value} of $\sigma(x)$.

To wit, let 
\begin{equation}
\sigma_{k}=\frac{\cos\left(\pi k/N_{\sigma}\right)+1}{2}\left(\sigma_{\max}-\sigma_{\min}\right)+\sigma_{\min},\quad k=0,\ldots,N_{\sigma},\label{eq:chebscheme}
\end{equation}
 denote the Chebyshev-Lobatto grid with $N_{\sigma}+1$ points on
the interval $[\sigma_{\min},\sigma_{\max}]$. Let $P_{k}$ be the
Lagrange interpolating polynomials for this grid, i.e., the polynomials
of degree $N_{\sigma}$ such that $P_{k}(\sigma_{j})=\delta_{jk}$.

Then we approximate $\mathcal{C}_{t}(x,z)$ via Chebyshev interpolation
as 
\begin{equation}
\mathcal{C}_{t}^{(N_{\sigma})}(x,z):=\sum_{k=0}^{N_{\sigma}}P_{k}(\sigma(x))\,\mathcal{G}_{k,t}(x,z),\label{eq:CtNsig}
\end{equation}
 in which each term 
\begin{equation}
\mathcal{G}_{k,t}(x,z):=e^{-\frac{\vert x-z\vert^{2}}{2\sigma_{k}^{2}\chi^{2}(t)}}\label{eq:Gkt}
\end{equation}
 is a \emph{bona fide} Gaussian convolution operator.

In turn we may define 
\begin{equation}
\mathcal{B}_{t}^{(N_{\sigma})}:=\mathcal{D}_{t}\,\mathcal{C}_{t}^{(N_{\sigma})},\label{eq:BtNsig}
\end{equation}
 which approximates $\mathcal{B}_{t}$. By lumping together the diagonal
multiplier $\mathcal{D}_{t}$ (\ref{eq:Dt}) with the diagonal multiplier
$P_{k}(\sigma(x))$, we may may define a diagonal operator
\begin{equation}
\mathcal{W}_{k}(x,z):=\underbrace{w(x)\left[2\pi\,\sigma^{2}(x)\right]^{-d/2}P_{k}(\sigma(x))}_{=:w_{k}(x)}\,\delta(x-z)\label{eq:Wk}
\end{equation}
 such that 
\begin{equation}
\mathcal{B}_{t}^{(N_{\sigma})}=\sum_{k=0}^{N_{\sigma}}\mathcal{W}_{k}\,\mathcal{G}_{k,t}.\label{eq:BWG}
\end{equation}

\subsection{Fourier discretization \label{sec:approxgrid}}

The preceding discussion motivates us to perform fast computation
with the operator $\mathcal{B}_{t}^{(N_{\sigma})}\mathcal{B}_{t}^{(N_{\sigma})*}$.
However, the internal `$dz$' integration implicit in the product
of these two integral operators must be discretized. This discretization
will be achieved by choosing a grid in Fourier space. One motivation
for considering a Fourier discretization is that the convolution operations
introduced above can naturally be performed as pointwise multiplications
in Fourier space. 

In the following, we let $\mathcal{F}$ denote the unitary Fourier
transform on $\R^{d}$. Then we approximate 
\[
\mathcal{B}_{t}^{(N_{\sigma})}\mathcal{B}_{t}^{(N_{\sigma})*}\approx\mathcal{B}_{t}^{(N_{\sigma})}\mathcal{\mathcal{F}}^{*}\Pi^{*}\Pi\mathcal{F}\mathcal{B}_{t}^{(N_{\sigma})*}\,(\Delta\omega)^{d},
\]
 where $\Pi=\Pi^{(M,\Delta\omega)}$ acts by restricting its input
to a discrete grid defined in terms of two parameters, $M$ and $\Delta\omega$.
Concretely, the restriction $\Pi\hat{f}$ of a function $\hat{f}$
to the grid can be viewed as a function of $\mathbf{n}\in\{-M,\ldots,M\}^{d}$
obtained as $(\Pi\hat{f})[\mathbf{n}]=\hat{f}(\mathbf{n}\,\Delta\omega)$.
We will typically omit the dependence of $\Pi$ on $M$ and $\Delta\omega$
from the notation to avoid notational clutter. 

Note that the formal adjoint $\Pi^{*}$ acts on grid functions $a=a[\mathbf{n}]$
via 
\[
\Pi^{*}a=\sum_{\mathbf{n}\in\{-M,\ldots,M\}^{d}}a[\mathbf{n}]\,\delta_{\mathbf{n}\,\Delta\omega},
\]
 with equality in the sense of distributions, so that 
\[
\Pi^{*}\Pi\hat{f}=\sum_{\mathbf{n}\in\{-M,\ldots,M\}^{d}}\hat{f}(\mathbf{n}\,\Delta\omega)\,\delta_{\mathbf{n}\,\Delta\omega}.
\]

\subsection{Summary \label{sec:approximationsummary}}

In summary, we propose to approximate the kernel $\mathcal{K}$ (\ref{eq:nonstat2})
as 
\begin{equation}
\tilde{\mathcal{K}}:=\sum_{j=0}^{N_{t}}\tilde{v}_{j}\,\mathcal{B}_{t_{j}}^{(N_{\sigma})}\mathcal{F}^{*}\Pi^{*}\Pi\mathcal{F}\mathcal{B}_{t_{j}}^{(N_{\sigma})^{*}},\label{eq:Ktilde}
\end{equation}
 in which expression we now define 
\[
\tilde{v}_{j}:=(\Delta\omega)^{d}v_{j}
\]
 to simplify our notation.

In total, the hyperparameters that must be chosen to define this approximation
are $t_{\min},t_{\max},N_{t}$ (for numerical integration in $t$,
cf. Section \ref{sec:approxt} above), $N_{\sigma}$ (for interpolation
in $\sigma$, cf. Section \ref{sec:approxsigma} above), and $M,\Delta\omega$
(for Fourier discretization, cf. Section \ref{sec:approxgrid} above).
We will analyze the error as a function of these choices in Section
\ref{sec:error} below. Before doing so, we will derive in Section
\ref{sec:algo} a fast algorithm for kernel matrix-vector multiplication
using the approximation $\tilde{\mathcal{K}}$. 

\section{Fast algorithm for kernel matrix-vector multiplication \label{sec:algo}}

For arbitrary scattered points $x_{1},\ldots,x_{N}\in\R^{d}$, define
the $N\times N$ kernel matrix 
\begin{equation}
\mathbf{K}=\left(\mathcal{K}(x_{i},x_{j})\right)_{i,j=1}^{N}.\label{eq:boldK}
\end{equation}
 Our goal is to perform the matrix-vector multiplication $\mathbf{K}\alpha$
with nearly-linear scaling in $N$, the size of the dataset, where
$\alpha=(\alpha_{j})\in\R^{N}$ is an arbitrary vector.

To approximate this result, we may likewise define 
\begin{equation}
\tilde{\mathbf{K}}=\left(\tilde{\mathcal{K}}(x_{i},x_{j})\right)_{i,j=1}^{N},\label{eq:boldKtilde}
\end{equation}
 where $\tilde{\mathcal{K}}$ is defined as in (\ref{eq:Ktilde})
and the dependence on our approximation hyperparameters is again omitted
for notational clarity. We will explain how to compute $\tilde{\mathbf{K}}\alpha$
\emph{exactly }(ignoring only, for simplicity, the numerical error
in the application of various NUFFTs). Note that the positive semidefiniteness
of the approximate kernel matrix $\tilde{\mathbf{K}}$ is automatically
guaranteed from the representation (\ref{eq:Ktilde}) of $\tilde{\mathcal{K}}$.

\REV{Our derivation will involve first rewriting \eqref{eq:Ktilde} so that the convolutions by individual Gaussians take place in Fourier space. We will then order the sums over our discrete values of $\sigma$ and $t$ to optimize the implementation. This will give us an expression for $\tilde{\mathbf{K}}\alpha$ from which we can simply read off our algorithm by applying its constituent operators from right to left.}

\subsection{Derivation}

For fixed $\alpha$, define the distribution 

\[
f:=\sum_{i=1}^{N}\alpha_{i}\delta_{x_{i}}.
\]
 The first step is to realize that 
\begin{equation}
\left[\tilde{\mathbf{K}}\alpha\right]_{i}=(\tilde{\mathcal{K}}f)(x_{i}).\label{eq:realize}
\end{equation}

Now, inserting (\ref{eq:BWG}) into our representation (\ref{eq:Ktilde})
of $\tilde{\mathcal{K}}$, we expand: 
\begin{equation}
\tilde{\mathcal{K}}f=\sum_{k'=0}^{N_{\sigma}}\,\mathcal{W}_{k'}\sum_{j=0}^{N_{t}}\tilde{v}_{j}\mathcal{G}_{k',t_{j}}\mathcal{F}^{*}\Pi^{*}\sum_{k=0}^{N_{\sigma}}\Pi\mathcal{F}\mathcal{G}_{k,t_{j}}\mathcal{W}_{k}f.\label{eq:oof}
\end{equation}
We will insert resolutions of the identity $\mathcal{F}\mathcal{F}^{*}=\mathrm{Id}$
and $\mathcal{F}^{*}\mathcal{F}=\mathrm{Id}$ into (\ref{eq:oof})
to obtain 
\begin{equation}
\tilde{\mathcal{K}}f=\sum_{k'=0}^{N_{\sigma}}\,\mathcal{W}_{k'}\sum_{j=0}^{N_{t}}\tilde{v}_{j}\mathcal{F}^{*}\hat{\mathcal{G}}_{k',t_{j}}\Pi^{*}\sum_{k=0}^{N_{\sigma}}\Pi\hat{\mathcal{G}}_{k,t_{j}}\mathcal{F}\mathcal{W}_{k}f,\label{eq:oof2}
\end{equation}
 where 
\[
\hat{\mathcal{G}}_{k,t}:=\mathcal{F}\mathcal{G}_{k,t}\mathcal{F}^{*}
\]
 is the Fourier-space representation of the Gaussian convolution operator,
i.e., a diagonal multiplier by a suitable Gaussian. We will also define
\[
\hat{G}_{k,t}:=\Pi\hat{\mathcal{G}}_{k,t}\Pi^{*}
\]
 denote the suitable restriction of $\mathcal{G}_{k,t}$ to a diagonal
multiplier on our $(M,\Delta\omega)$-Fourier grid. Note that in fact
\[
\Pi\hat{\mathcal{G}}_{k,t}=\hat{G}_{k,t}\Pi,\quad\quad\hat{\mathcal{G}}_{k,t}\Pi^{*}=\Pi^{*}\hat{G}_{k,t}
\]
 from which facts, together with (\ref{eq:oof2}), it follows that
\begin{equation}
\tilde{\mathcal{K}}f=\sum_{k'=0}^{N_{\sigma}}\,\mathcal{W}_{k'}[\mathcal{F}\Pi]^{*}\overbrace{\sum_{j=0}^{N_{t}}\tilde{v}_{j}\hat{G}_{k',t_{j}}\sum_{k=0}^{N_{\sigma}}\hat{G}_{k,t_{j}}\underbrace{\Pi\mathcal{F}\left[\mathcal{W}_{k}f\right]}_{\text{type 1 NUFFT}}}^{\text{Fourier grid function}\ a_{k}=a_{k}[\mathbf{n}]}.\label{eq:oof3}
\end{equation}
 Observe that the underbraced expression, which is a function on the
Fourier grid, can be constructed exactly as the type 1 NUFFT (cf.
Appendix \ref{app:nuffts}, and recall from (\ref{eq:Wk}) the definition
of the diagonal multiplier $\mathcal{W}_{k}$) of the vector $c_{k}\in\R^{N}$
of values 
\begin{equation}
c_{k,i}:=w_{k}(x_{i})\,\alpha_{i},\label{eq:c}
\end{equation}
 associated to the scattered points $\{x_{i}\}_{i=1}^{N}$.

Using results of these $(N_{\sigma}+1)$ type 1 NUFFTs, we can form
the Fourier grid function $a_{k}=a_{k}[\mathbf{n}]$ indicated with
the overbrace in (\ref{eq:oof3}) by a sequence of pointwise multiplications
and summations. Given the result $a_{k}$, we now indicate the computation
that remains for determining $\left[\tilde{\mathbf{K}}\alpha\right]_{i}$,
following \ref{eq:realize}: 
\begin{equation}
\left[\tilde{\mathbf{K}}\alpha\right]_{i}=(\tilde{\mathcal{K}}f)(x_{i})=\sum_{k=0}^{N_{\sigma}}\,w_{k}(x_{i})\underbrace{\left([\mathcal{F}\Pi]^{*}a_{k}\right)(x_{i})}_{\text{type 2 NUFFT}}.\label{eq:oof4}
\end{equation}
 Observe that the collection of values $\left\{ \left([\mathcal{F}\Pi]^{*}a_{k}\right)(x_{i})\right\} _{i=1}^{N}$
indicated with the underbrace can be recovered precisely as the type
2 NUFFT (cf. Appendix \ref{app:nuffts}) of the Fourier grid function
$a_{k}$, evaluated on the scattered points $\{x_{i}\}_{i=1}^{N}$.

\subsection{Explicit algorithm and cost scaling}

We concretely summarize the algorithmic steps indicated above in the
discussion of (\ref{eq:oof3}) and (\ref{eq:oof4}).
\begin{enumerate}
\item [(1)] For each $k=0,\ldots,N_{\sigma}$, form the vector $c_{k}\in\R^{N}$,
following (\ref{eq:c}).
\begin{itemize}
\item [--]\textbf{ Cost scaling:} $O(N_{\sigma}N)$.
\end{itemize}
\item [(2)] For each $k=0,\ldots,N_{\sigma}$, form the Fourier grid function
$b_{k}=b_{k}[\mathbf{n}]$ as the type 1 NUFFT of $c_{k}$ on the
scattered points $\{x_{i}\}_{i=1}^{N}$. (These $(N_{\sigma}+1)$
type 1 NUFFTs can be batched in parallel.)
\begin{itemize}
\item [--]\textbf{ Cost scaling: }$O(N_{\sigma}M^{d}+N_{\sigma}N\log N)$,
cf. Appendix \ref{app:nuffts}.
\end{itemize}
\item [(3)] For each $k'=0,\ldots,N_{\sigma}$, form $a_{k'}:=\sum_{j=0}^{N_{t}}\tilde{v}_{j}\hat{G}_{k',t_{j}}\sum_{k=0}^{N_{\sigma}}\hat{G}_{k,t_{j}}b_{k}$.
\begin{itemize}
\item [--]\textbf{ Cost scaling:} $O(M^{d}N_{\sigma}\min(N_{\sigma},N_{t}))$,
cf. Remark \ref{rem:step3} below for further detail.
\end{itemize}
\item [(4)] For each $k=0,\ldots,N_{\sigma}$, form $\beta_{k}\in\R^{N}$
as the type 2 NUFFT of the the Fourier grid function $a_{k}$, evaluated
on the scattered points $\{x_{i}\}_{i=1}^{N}$. (These $(N_{\sigma}+1)$
NUFFTs can also be batched.)
\begin{itemize}
\item [--]\textbf{ Cost scaling:} $O(N_{\sigma}M^{d}+N_{\sigma}N\log N)$,
cf. Appendix \ref{app:nuffts}.
\end{itemize}
\item [(5)] Then $\tilde{\mathbf{K}}\alpha$ is finally recovered as $\sum_{k=0}^{N_{\sigma}}\,w_{k}(x_{i})\beta_{k}$.
\begin{itemize}
\item [--]\textbf{ Cost scaling: }$O(N_{\sigma}N)$. Note that $w_{k}(x_{i})$
can be precomputed in step 1.
\end{itemize}
\end{enumerate}
\vspace{2mm}
\begin{rem}
\emph{\label{rem:step3}Note that performing step $3$ directly requires
us to perform one matrix-vector multiplication of size $(N_{t}+1)\times(N_{\sigma}+1)$
and another of size $(N_{\sigma}+1)\times(N_{t}+1)$, for each point
on the Fourier grid. In total the cost scaling of this approach amounts
to $O(M^{d}N_{\sigma}N_{t})$.}

\emph{Alternatively, we may precompute the tensor 
\[
A_{k',k}[\mathbf{n}]=\sum_{j=0}^{N_{t}}\tilde{v}_{j}\hat{G}_{k',t_{j}}[\mathbf{n},\mathbf{n}]\,\hat{G}_{k,t_{j}}\,[\mathbf{n},\mathbf{n}]
\]
 with offline cost $O(M^{d}N_{\sigma}^{2}N_{t})$. Then once this
tensor is formed, we may form the Fourier grid functions $a_{k'}$
in terms of the $b_{k}$ as 
\[
a_{k'}[\mathbf{n}]=A_{k',k}[\mathbf{n}]\,b_{k}[\mathbf{n}]
\]
 with online cost $O(M^{d}N_{\sigma}^{2})$. This is useful when $(N_{t}+1)\geq(N_{\sigma}+1)/2$
since we need to perform only a single matrix-vector multiplication
of size $(N_{\sigma}+1)\times(N_{\sigma}+1)$ for each Fourier grid
point.}

\emph{Even when $N_{\sigma}$ becomes large, we comment that precomputation
may still be useful if it is possible to factorize each $A_{k',k}[\mathbf{n}]\approx\sum_{\alpha=1}^{r}R_{k'}^{\alpha}[\mathbf{n}]R_{k}^{\alpha}[\mathbf{n}]$
in low-rank form. This could allow the user to set $N_{t}$ very large,
oversampling the integral in $t$, and then reveal the rank that is
empirically required, rather than fixing it }a priori\emph{. We leave
further investigation of this point to future work.}
\end{rem}

\vspace{2mm}

In summary, the total cost scaling of a matvec is 
\begin{equation}
\label{eq:complexity}
O(N_{\sigma}N\log N+M^{d}N_{\sigma}\min(N_{\sigma},N_{t})),
\end{equation}
 though refer to Remark \ref{rem:step3} for a discussion of a potential
offline cost.

\REV{We comment that the fast Gauss transform (FGT) \cite{Greengard_Strain_1991, Spivak_Veerapaneni_Greengard_2010, Greengard_Jiang_Rachh_Wang_2024} might also be used to accelerate the applications of $\mathcal{G}_{k, t}$. However, we would have to run one FGT for each of the $O(N_tN_\sigma)$ Gaussians by which wish to convolve. Since the FGT cost scales linearly in the number of input points, an FGT-based algorithm would have an $O(N_tN_\sigma N)$ term in its cost-scaling expression.}

\section{Error analysis \label{sec:error}}

In Section \ref{sec:algo}, we presented a fast algorithm for matrix-vector
multiplication by the approximate kernel matrix $\tilde{\mathbf{K}}$
(\ref{eq:boldKtilde}). In this section, we want to bound the error
compared to multiplication by the true kernel matrix $\mathbf{K}$
(\ref{eq:boldKtilde})

\REV{We will accomplish this by bounding the error $\trip \mathcal{K}-\tilde{\mathcal{K}} \trip$ in a suitable norm $\trip \,\cdot \, \trip$, which in turn controls the error $\Vert\mathbf{K} - \tilde{\mathbf{K}}\Vert_p$ for any operator $p$-norm. We bound the norm $\trip \mathcal{K}-\tilde{\mathcal{K}} \trip$ in three stages:
\begin{itemize}
\item [(1)] bounding the error due to numerical integration in $t$, cf.
Section \ref{sec:approxt} above,
\item [(2)] bounding the error due to interpolation in $\sigma$, cf. Section
\ref{sec:approxsigma} above, and 
\item [(3)] bounding the error due to Fourier discretization, cf. Section
\ref{sec:approxgrid} above.
\end{itemize}}

For simplicity we will assume that our scattered points $\{x_{i}\}_{i=1}^{N}$
are contained in the bounding box $[-1,1]^{d}$. This assumption does
not lose any generality, as we can reduce to this scenario by shifting
and scaling the problem.

Now observe that for any $\alpha\in\R^{N}$ and any $p\in[1,\infty]$
\[
\Vert\mathbf{K}\alpha-\tilde{\mathbf{K}}\alpha\Vert_{p}\leq\Vert\mathbf{K}-\tilde{\mathbf{K}}\Vert_{p}\Vert\alpha\Vert_{p},
\]
 where $\Vert\,\cdot\,\Vert_{p}$ denotes the $\ell^{p}$ vector norm
as well as the corresponding operator norm. In fact, it will be most
convenient to bound the error $\mathbf{K}-\tilde{\mathbf{K}}$ in
the `uniform entrywise' norm, which is not an operator norm.
\begin{defn}
\emph{\label{def:triplenorm} For an $N\times N$ matrix $A$, we
define 
\[
\trip A\trip=\max_{i,j\in\{1,\ldots,N\}}\vert A_{ij}\vert.
\]
 Likewise, for an integral kernel $\mathcal{A}(x,y)$ which is a continuous
function of $(x,y)$, we define (overloading notation slightly) the
analogous norm 
\[
\trip\mathcal{A}\trip:=\sup_{x,y\in[-1,1]^{d}}\vert\mathcal{A}(x,y)\vert.
\]
}

Fortunately, the uniform entrywise norm on matrices controls all $p$-operator
norms as follows: 
\end{defn}

\begin{lem}
\label{lem:holder} For any $N\times N$ matrix $A$ and any $p\in[1,\infty]$,
we have $\Vert A\Vert_{p}\leq N\,\trip A\trip$.
\end{lem}

\noindent The proof is given in Appendix \ref{app:lemmas} for completeness,
since this fact is not too frequently encountered.

Following Lemma \ref{lem:holder}, as well as the immediate fact that
$\trip\mathbf{K}-\tilde{\mathbf{K}}\trip\leq\trip\mathcal{K}-\tilde{\mathcal{K}}\trip$,
the error of our kernel matrix-vector multiplication algorithm can
be bounded as 
\[
\Vert\mathbf{K}\alpha-\tilde{\mathbf{K}}\alpha\Vert_{p}\leq N\,\trip\mathcal{K}-\tilde{\mathcal{K}}\trip
\]
 for any $p\in[1,\infty]$.

Thus\REV{, as indicated above,} we are motivated to bound $\trip\mathcal{K}-\tilde{\mathcal{K}}\trip$
. For simplicity we will assume that 
\begin{equation}
\vert w(x)\vert\leq1,\ \text{for all }x\in[-1,1]^{d}.\label{eq:w_assumption}
\end{equation}
 Evidently, the entrywise norm $\,\trip\mathcal{K}-\tilde{\mathcal{K}}\trip$
scales more generally with an additional factor factor of $\Vert w\Vert_{\infty}^{2}$
where $\Vert w\Vert_{\infty}$ indicates the $L^{\infty}$ norm
on the bounding box. We reduce to the case $\Vert w\Vert_{\infty}\leq1$
to avoid notational clutter. 

In our error analysis we will view $t_{\min}$, $t_{\max}$, $N_{t}$,
$N_{\sigma}$, $M$, and $\Delta\omega$ as the adjustable parameters,
cf. the summary of these choices in Section \ref{sec:approximationsummary}.
Meanwhile the dimension $d$ and the Mat\'ern parameter $\nu$ (\ref{eq:matern})
will be viewed as constant from the point of view of big-$O$ notation.
A glossary of commonly used notation in our analysis is provided in
Appendix \ref{app:glossary}. 

\subsection{Stage (1)}

The following lemma constitutes stage (1)\REV{, in which we bound the error introduced by discretizing the integral in $t$.}

\begin{lem}
\label{lem:dt} For the Mat\'ern function representation with $u$ and
$\chi$ as in (\ref{eq:maternpres}), the following bound holds for
any choice of $\alpha>0$ and $\delta\in(0,1)$: 
\[
\lrtrip{\mathcal{K}-\sum_{j=0}^{N_{t}}\mathcal{B}_{t_{j}}\mathcal{B}_{t_{j}}^{*}\,v_{j}}=O\left(\sigma_{\min}^{-d}\left[e^{-\alpha t_{\max}}+e^{\nu t_{\min}}+e^{-\frac{(1-\delta)\pi^{2}}{\Delta t}}\right]\right),
\]
 where $\Delta t=\frac{t_{\max}-t_{\min}}{N_{t}}$ is assumed to be
$O(1)$. 
\end{lem}

\noindent The proof is given in Appendix \ref{app:lemmas}. \REV{In brief, the third term in the error bound appearing in the statement of the lemma is due to the approximation of the integral in $t$ with an infinite Riemann sum from $t=-\infty$ to $t=\infty$. The first and second terms, meanwhile, are due to the truncation of the tails of this Riemann sum outside of $[ t_{\min} , t_{\max} ]$.}

Motivated by the result of Lemma \ref{lem:dt}, we make the following definition. 
\begin{defn}
\emph{\label{def:eptrap} For any $t_{\min},t_{\max}\in\R$, $N_{t}\geq1$,
and $\nu,\alpha>0$, $\delta\in(0,1)$, define 
\[
\epsilon_{\mathrm{trap}}^{\nu,\alpha,\delta}\left(t_{\min},t_{\max,}N_{t}\right)\coloneqq e^{-\alpha t_{\max}}+e^{\nu t_{\min}}+e^{-\frac{(1-\delta)\pi^{2}}{\Delta t}}.
\]
Usually we omit the dependence on $\nu,\alpha,\delta$ (which we can
take to be fixed), as well as on $t_{\min},$ $t_{\max},$ and $N_{t}$
(our adjustable parameters) from the notation, simply writing $\epsilon_{\mathrm{trap}}$.}
\end{defn}

\subsection{Stage (2)}

To accomplish stage (2), \REV{bounding the error due to Chebyshev interpolation in $\sigma$,} first we show that Chebyshev interpolation
of the Gaussian function in the \emph{width} parameter attains uniform
accuracy that is controlled only by $\kappa=\frac{\sigma_{\max}}{\sigma_{\min}}$
and the number of interpolating widths $N_{\sigma}$.
\begin{lem}
\label{lem:interp}For all $x\in\R$, 
\[
\left|e^{-\frac{x^{2}}{2\sigma^{2}}}-\sum_{k=1}^{N_{\sigma}}P_{k}(\sigma)e^{-\frac{x^{2}}{2\sigma_{k}^{2}}}\right|\leq2(\kappa-1)\left(\frac{\kappa-1}{\kappa+1}\right)^{N_{\sigma}}.
\]
\end{lem}

\noindent The proof is given in Appendix \ref{app:lemmas} \REV{and makes use of standard error bounds for Chebyshev interpolation of analytic functions \cite{trefethen_approximation_2020}}. Motivated
by the result of Lemma \ref{lem:interp}, we make the following definition. 
\begin{defn}
\emph{\label{def:epcheb}For any $\kappa\geq1$ and integer $N_{\sigma}\geq1$,
define 
\[
\epsilon_{\mathrm{cheb}}(N_{\sigma},\kappa)=2(\kappa-1)\left(\frac{\kappa-1}{\kappa+1}\right)^{N_{\sigma}}.
\]
 As before, we will typically omit the dependence on $N_{\sigma}$
and $\kappa$ from the notation, simply writing $\epsilon_{\mathrm{cheb}}$.}
\end{defn}

We also make a few more definitions that are necessary to state the
bound that finishes stage (2).
\begin{defn}
\emph{Let $\rho_{\max}=\sigma_{\max}\chi_{\max}$ and $\rho_{\min}=\sigma_{\min}\chi_{\min}$,
where $\chi_{\max}:=\chi(t_{\max})$ and $\chi_{\min}:=\chi(t_{\min})$.
Furthermore, let $\lambda=\frac{\rho_{\max}}{\rho_{\min}}=\kappa\frac{\chi_{\max}}{\chi_{\min}}$.}
\end{defn}

\begin{defn}
\emph{\label{def:lebesgue}Let $\Lambda_{N_{\sigma}}$ denote the
Lebesgue constant \cite{trefethen_approximation_2020} for Chebyshev
interpolation on the nodes $\sigma_{k}$, i.e., 
\[
\Lambda_{N_{\sigma}}:=\sup_{\sigma\in[\sigma_{\min},\sigma_{\max}]}\ \left\{ \sum_{k=0}^{N_{\sigma}}\left|P_{k}(\sigma)\right|\right\} .
\]
It is known \cite[Theorem 15.2]{trefethen_approximation_2020} that
$\Lambda_{N_{\sigma}}=O\left(\log N_{\sigma}\right)$.}
\end{defn}

Then Lemma \ref{lem:interp} allows us to prove the following bound,
which controls the error incurred in the kernel by our Chebyshev interpolation
procedure.
\begin{lem}
\label{lem:sigma} The bound 
\[
\lrtrip{\mathcal{B}_{t}\mathcal{B}_{t}^{*}-\mathcal{B}_{t}^{(N_{\sigma})}\mathcal{B}_{t}^{(N_{\sigma})*}}=O\left(\Lambda_{N_{\sigma}}\,\rho_{\max}^{d}\,\epsilon_{\mathrm{cheb}}\right)
\]
 holds uniformly in $t\in[t_{\min},t_{\max}]$.
\end{lem}

\noindent The proof is given in Appendix \ref{app:lemmas}.

\subsection{Stage (3)}

Finally, to accomplish stage (3), we want to bound the error incurred
by the `insertion' of $(\Delta\omega)^{d}\mathcal{F}^{*}\Pi^{*}\Pi\mathcal{F}$
within the product $\mathcal{B}_{t}^{(N_{\sigma})}\mathcal{B}_{t}^{(N_{\sigma})*}$.
A first step to achieving this bound is the following lemma.
\begin{lem}
\label{lem:stage3warmup}The bound 
\[
\trip(\Delta\omega)^{d}\,\mathcal{G}_{k,t}\mathcal{F}^{*}\Pi^{*}\Pi\mathcal{F}\mathcal{G}_{l,t}-\mathcal{G}_{k,t}\mathcal{G}_{l,t}\trip=O\left(\rho_{\max}^{d}\left[\lambda^{d}\,e^{-\left[2\pi\rho_{\min}M\Delta\omega\right]^{2}}+e^{-\left[\frac{1}{4\rho_{\max}\Delta\omega}\right]^{2}}\right]\right)
\]
 holds uniformly over $t\in[t_{\min},t_{\max}]$ and $k,l\in\{0,\ldots,N_{\sigma}\}$,
provided $\Delta\omega\leq\frac{1}{8}$.
\end{lem}

\noindent The proof is based on bounding the Riemann sum approximation
of the integral of a Gaussian using Poisson summation, and it is very
similar to Theorem 2 of \cite{barnett_uniform_2023}, for example.
See also \cite{trefethen_traprule_2014}. For completeness, we prove
the result from scratch in Appendix \ref{app:lemmas}.

Motivated by the result of Lemma \ref{lem:stage3warmup}, we make
the following definition. 
\begin{defn}
\label{def:epF} \emph{In terms of our adjustable hyperparameters
and their dependent quantities, define 
\[
\epsilon_{\mathcal{F}}:=\lambda^{d}\,e^{-[2\pi\rho_{\min}M\Delta\omega]^{2}}+e^{-\left[\frac{1}{4\rho_{\max}\Delta\omega}\right]^{2}},
\]
 where as above we omit the dependence on the parameters from our
notation.}
\end{defn}

Then the desired bound for stage (3) follows: 
\begin{lem}
\label{lem:stage3} The bound 
\[
\trip(\Delta\omega)^{d}\,\mathcal{B}_{t}^{(N_{\sigma})}\mathcal{F}^{*}\Pi^{*}\Pi\mathcal{F}\mathcal{B}_{t}^{(N_{\sigma})*}-\mathcal{B}_{t}^{(N_{\sigma})}\mathcal{B}_{t}^{(N_{\sigma})*}\trip=O\left(N_{\sigma}^{2}\,\kappa^{d}\left[\frac{\chi_{\max}}{\sigma_{\min}}\right]^{d}\epsilon_{\mathcal{F}}\right)
\]
 holds uniformly in $t\in[t_{\min},t_{\max}]$.
\end{lem}

\noindent The proof is given in Appendix \ref{app:lemmas}. 

\subsection{Synthesizing the bounds \label{sec:maintheorem}}

Now we can synthesize stages (1), (2), and (3) of the error analysis
into our main approximation theorem.
\begin{thm}
\label{thm:main} For the non-stationary isotropic kernel $\mathcal{K}$
(\ref{eq:nonstat2}) of Mat\'ern type (with parameter $\nu$) and approximation
$\tilde{\mathcal{K}}$ (\ref{eq:Ktilde}), the following bound holds:
\[
\trip\mathcal{K}-\tilde{\mathcal{K}}\trip=O\left(\sigma_{\min}^{-d}\,\epsilon_{\mathrm{trap}}+\Lambda_{N_{\sigma}}\,\rho_{\max}^{d}\,\epsilon_{\mathrm{cheb}}+N_{\sigma}^{2}\,\kappa^{d}\left[\frac{\chi_{\max}}{\sigma_{\min}}\right]^{d}\epsilon_{\mathcal{F}}\right),
\]
 provided that $\Delta t=O(1)$, $\Delta\omega\leq\frac{1}{8}$. In
the case of the squared-exponential kernel, the bound instead reads
as 
\[
\trip\mathcal{K}-\tilde{\mathcal{K}}\trip=O\left(\Lambda_{N_{\sigma}}\,\rho_{\max}^{d}\,\epsilon_{\mathrm{cheb}}+N_{\sigma}^{2}\,\kappa^{d}\left[\frac{\chi_{\max}}{\sigma_{\min}}\right]^{d}\epsilon_{\mathcal{F}}\right).
\]
\end{thm}

\noindent The proof is given in Appendix \ref{app:main}. Please
refer to Definitions \ref{def:eptrap}, \ref{def:epcheb}, and \ref{def:epF}
for the definitions of $\epsilon_{\mathrm{trap}}$, $\epsilon_{\mathrm{cheb}}$,
and $\epsilon_{\mathcal{F}}$. For other quantities, the glossary
of Appendix \ref{app:glossary} may provide a helpful supplement to
the above discussion. 

\REV{Given a user-selected tolerance $\epsilon$ for the pointwise error $\trip \mathcal{K} - \tilde{\mathcal{K}} \trip$,} Theorem \ref{thm:main} offers a guide for selecting the adjustable
parameters $t_{\min}$, $t_{\max}$, $N_{t}$, $N_{\sigma}$, $M$,
and $\Delta\omega$. \REV{Recall that the parameters fixed by the problem specification are $d$, $\nu$, $\sigma_{\min}$, and $\sigma_{\max}$. The reader is referred to Appendix~Appendix \ref{app:glossary} for the glossary of notation. Since our analysis is written in terms of big-$O$ notation, we cannot make explicit choices that will achieve an error of at most $\epsilon$.  Nonetheless, we can offer a practical perspective on the procedure for choosing the parameters in terms of $\epsilon$.}
The discussion
is only informal, and the reader is encouraged to refer to the rigorous
statement of Theorem \ref{thm:main}.
\REV{
\begin{enumerate}[label={(\arabic*)}]
\item First we choose $t_{\min}$, $t_{\max}$, and $N_{t}$ to control $\epsilon_{\textrm{trap}}$ so that the first term in Theorem \ref{thm:main} is bounded as $ \sigma_{\min}^{-d} \epsilon_{\textrm{trap}} = O(\epsilon)$.
Based on Definition \ref{def:eptrap}, concretely we take 
\[
-t_{\min} = \nu^{-1} \log(1/\epsilon) + \nu^{-1} d \log(1/\sigma_{\min}), 
\]
Similarly, we take
\[
t_{\max} = \alpha^{-1} \log(1/\epsilon) + \alpha^{-1} d \log(1/\sigma_{\min}) + c_\alpha,
\]
where $\alpha \geq 1$ is arbitrary and $c_\alpha$ is universal other than its dependence on $\alpha$. Then given $t_{\min}$ and $t_{\max}$, we choose 
\[
N_{t} = (t_{\max}-t_{\min}) \log(1/\epsilon)
\]
\begin{itemize}
\item Now that $t_{\min}$ and $t_{\max}$ are fixed, the quantities $\lambda$,
$\chi_{\min}$, $\chi_{\max}$, $\rho_{\min}$, and $\rho_{\max}$
are also fixed.  
\end{itemize}
\item Then we choose $N_{\sigma}$ to  to control $\epsilon_{\textrm{cheb}}$ so that the second term in Theorem \ref{thm:main} is bounded as $ \Lambda_{N_{\sigma}} \rho_{\max}^{d} \epsilon_{\textrm{cheb}} \leq \epsilon$.
Based on Definition \ref{def:epcheb}, we take 
\[
N_{\sigma} = O( \kappa \,  [ \log (1/\epsilon) + \log(\sigma_{\max}) ] ).
\]
We comment that in the large $\kappa$ and small $\epsilon$ limit, asymptotically we can take $N_{\sigma} \sim \kappa\log (1/\epsilon) $.
\item Finally, we choose $M$ and $\Delta\omega$ to control $\epsilon_{\mathcal{F}}$ so that the third term in Theorem \ref{thm:main} is bounded as $N_{\sigma}^{2}\kappa^{d}\left[\frac{\chi_{\max}}{\sigma_{\min}}\right]^{d} \epsilon_{\mathcal{F}} = O(\epsilon)$.
Based on Definition \ref{def:epF}, we take 
\[
(\Delta\omega)^{-1} = \rho_{\max}\sqrt{\log(1/\epsilon) + \ldots }
\]
where the omitted terms are $O(d \,  [\log \kappa + \log(1/\sigma_{\min}) + \log \log  (1/\epsilon) + \log_+ (1/\nu) ])$. Then in terms of $\Delta \omega$, we choose
\[
M = \frac{(\Delta\omega)^{-1}}{\rho_{\min}}\sqrt{ (1+\nu^{-1}) \log (1/\epsilon) + \ldots}
\]
where the omitted terms are $O(d \,  [\log \kappa + \log_+ (1/\nu) ])$.
\end{enumerate}
With these choices, the error $\trip \mathcal{K} - \tilde{\mathcal{K}} \trip$ will be $O(\epsilon)$. With all else held constant, we may determine the computational complexity in terms of the error tolerance $\epsilon$ and  Mat\'ern parameter $\nu$ by taking $N_t=O( (\nu^{-1} + \alpha^{-1} )  \log^2(1/\epsilon))$, $N_\sigma=O( \log( 1/ \epsilon) )$, and
\[
M = \frac{\rho_{\max}}{\rho_{\min}} \, O \left( \sqrt{1+\nu^{-1}} \, \log(1/\epsilon) \right).
\]
After observing that  $\rho_{\max} / \rho_{\min} = \kappa \, e^{ \frac{1}{2} ( t_{\max} - t_{\min} )}$ and plugging in our choices for $t_{\max}$ and $t_{\min}$, we see that we can take
\[
M = O (\epsilon^{-p})
\]
where $p$ is any exponent larger than $(2 \nu)^{-1}$. By substituting into \eqref{eq:complexity}, we recover the final complexity $O(\epsilon^{-d p})$ for any such $p$. In particular, note that the requirements on the resolution of the Fourier grid become milder in the limit $\nu \ra \infty$ of a smoother  Mat\'ern prior. In Section \ref{sec:numerics} below, we will numerically explore
the effect of the tunable parameters on the accuracy of kernel matrix-vector
multiplications.}

\section{Numerical experiments\label{sec:numerics}}

Finally we present several numerical experiments validating the accuracy
and performance of our approach. In Section \ref{sec:matvecs} we
explore the impact of our tunable parameters on the error control
of an isolated kernel matvec. In Section \ref{sec:solves} we wrap
our matvec algorithm with a CG solver to solve a GPR problem and compare
the scaling against the FLAM package \cite{ho_flam_2020}. All experiments
were performed on a laptop with a 12th-generation Intel Core i7.

Throughout, we define a tolerance $\epsilon_{\textrm{NUFFT}}=10^{-6}$ and take NUFFTs
with a relative error tolerance of $\epsilon_{\textrm{NUFFT}}/10$. For a given maximum
standard deviation $\rho_{\max}$, we will take $\Delta\omega=\min\left\{ \frac{1}{8},\frac{1}{4}\rho_{\max}^{-1}\log^{-1/2}(1/\epsilon_{\textrm{NUFFT}})\right\} $,
which guarantees in particular (cf. Definition \ref{def:epF}) that
$e^{-\left[\frac{1}{4\rho_{\max}\Delta\omega}\right]^{2}}\leq\epsilon_{\textrm{NUFFT}}$.
All relative errors for vectors are measured in the 2-norm.

In all experiments, the points $x_{i}$, $i=1,\ldots,N$, are chosen
independently from the uniform distribution on $[-1,1]^{d}$. We also
fix $w(x)=1$ and $\sigma(x)=\frac{1}{6}\left(\prod_{i=1}^{d}\cos(\pi x_{i})+2\right)$,
except where otherwise indicated.

Code for all of our experiments is available at \cite{ourcode}.

\subsection{Matvecs \label{sec:matvecs}}

In this section we focus on the accuracy and cost of the computation
of a single matvec $\mathbf{K}\alpha$. 

First we fix the Mat\'ern parameter $\nu=\frac{3}{2}$ and set $t_{\min}=\left(1+\log\epsilon_{\textrm{NUFFT}}\right)/\nu$
and $t_{\max}=\log\left(-2\log\epsilon_{\textrm{NUFFT}}\right)$. These values are
chosen to be sufficiently small and large, respectively, such that
they do not bottleneck the error in our experiments. From these choices
we obtain $\chi_{\min}=\frac{1}{\sqrt{\nu}}e^{t_{\min}/2}$ and $\chi_{\max}=\frac{1}{\sqrt{\nu}}e^{t_{\max}/2}$.
We also choose $\sigma_{\min}=\frac{1}{6}-0.01$ and $\sigma_{\max}=\frac{1}{2}+0.01$,
in turn fixing the values of $\rho_{\min}$ and $\rho_{\max}$.

\REV{Throughout all our experiments, we maintain this value of $\kappa = \sigma_{\max} / \sigma_{\min} \approx 3$.  In practice, $\kappa$ can vary widely: in a large geographical model, correlation lengths might vary from a few miles to a few hundred miles, yielding $\kappa \approx 100$ or even larger \cite{Noack_Sethian_2022}. In general, $\kappa$ will depend on the range of length scales present within the model.}

Our tunable parameters are then $N_{t}$, $N_{\sigma}$, and
$M$. We perform ablation experiments, modifying $N_{t}$, $N_{\sigma}$,
and $M$ individually with all the remaining parameters fixed. The
results of these experiments are presented for $d=1$ and $d=2$,
respectively, in Figures \ref{fig:matvec1d} and \ref{fig:matvec2d}.
We take $N=10^{4}$ throughout; reference values for the other fixed
parameters can be found in the captions.

\begin{figure}
\centering{}\includegraphics[bb=0bp 10bp 461bp 346bp,clip,scale=0.6]{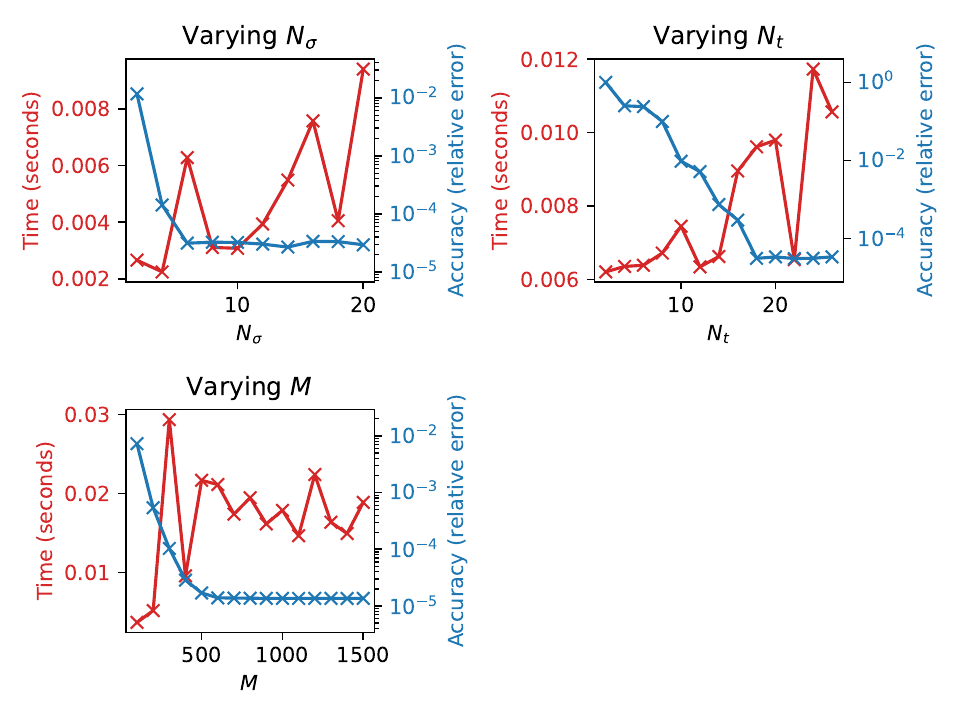}\caption{\label{fig:matvec1d}Matvec time and accuracy in one dimension for
the Mat\'ern kernel with $\nu=\frac{3}{2}$. Unless otherwise specified,
$N_{t}=20$, $N_{\sigma}=20$, and $M=400$.  \REV{Results are averaged over 50 runs.}}
\end{figure}

\begin{figure}
\centering{}\includegraphics[bb=0bp 10bp 461bp 346bp,clip,scale=0.6]{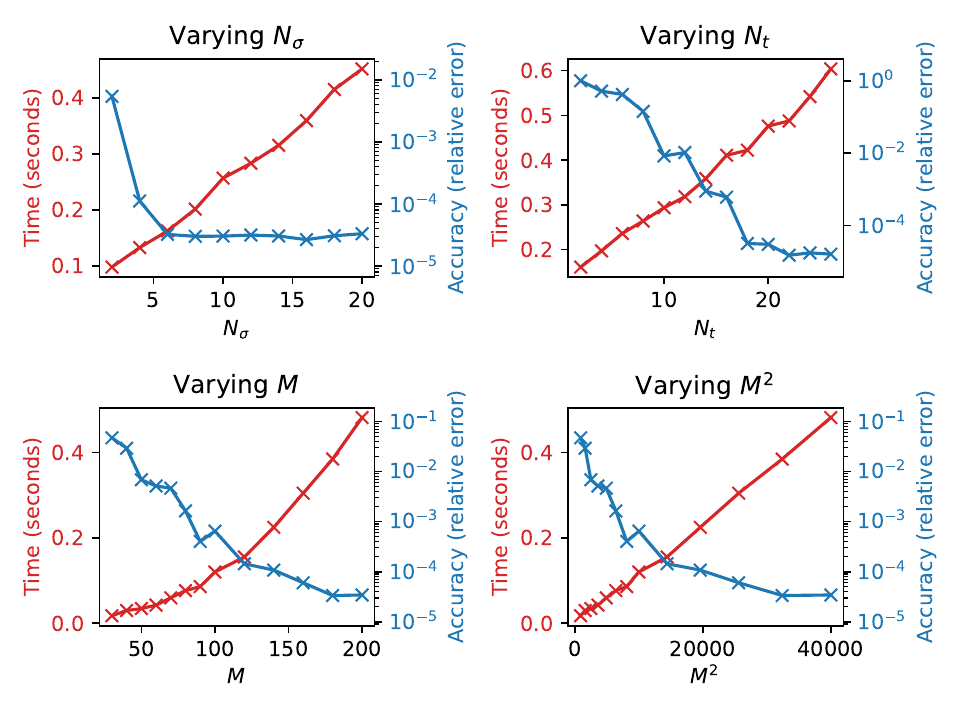}\caption{\label{fig:matvec2d}Matvec time and accuracy in two dimensions for
the Mat\'ern kernel with $\nu=\frac{3}{2}$. Unless otherwise specified,
$N_{t}=20$, $N_{\sigma}=20$, and $M=200$ ($M^{2}=40\,000$). \REV{Results are averaged over 3 runs.}}
\end{figure}

As expected, the time taken \REV{generally} grows linearly in $N_{t}$, $N_{\sigma}$,
and $M^{d}$, while errors decrease exponentially with $N_{t}$, $N_{\sigma}$,
and $M$. \REV{In the one-dimensional case, our accuracy saturates with respect to $M$ before the large-$M$ asymptotic cost scaling is revealed, so the performance appears as roughly constant in $M$ in these experiments.}  As we increase one parameter individually, the error will
eventually saturate. This is consistent with our understanding in
Theorem \ref{thm:main}, in which each of these parameters contributes
roughly independently to the error in an additive fashion. In the
one-dimensional case especially, there is some variance in the time
taken; this is because the matvec is sufficiently fast that smaller
effects such as constant setup times fail to fully amortize.

Finally, maintaining the parameters at their given reference values
for each case $d=1,2$, we vary $N$. The results are presented in
Figure \ref{fig:matvecN-matern}, showing approximately linear scaling
independent of dimension and a constant relative error. The reference
values for $N_{t}$, $N_{\sigma}$, and $M$ are the same as in Figures
\ref{fig:matvec1d} and \ref{fig:matvec2d}.

We comment that in higher dimensions, the $O(M^{d}N_{t}N_{\sigma})$
memory bottleneck of our most basic implementation becomes nontrivial.
This bottleneck owes to the formation of the tensor $\hat{G}_{k,t_{j}}[\mathbf{n},\mathbf{n}]$
(cf. Remark \ref{rem:step3}). A sequential implementation, in which
the sums over $k$ and $j$ are processed sequentially, could reduce
the memory bottleneck to $O(M^{d})$, but we do not pursue such modifications.

\begin{figure}
\centering{}\includegraphics[scale=0.5]{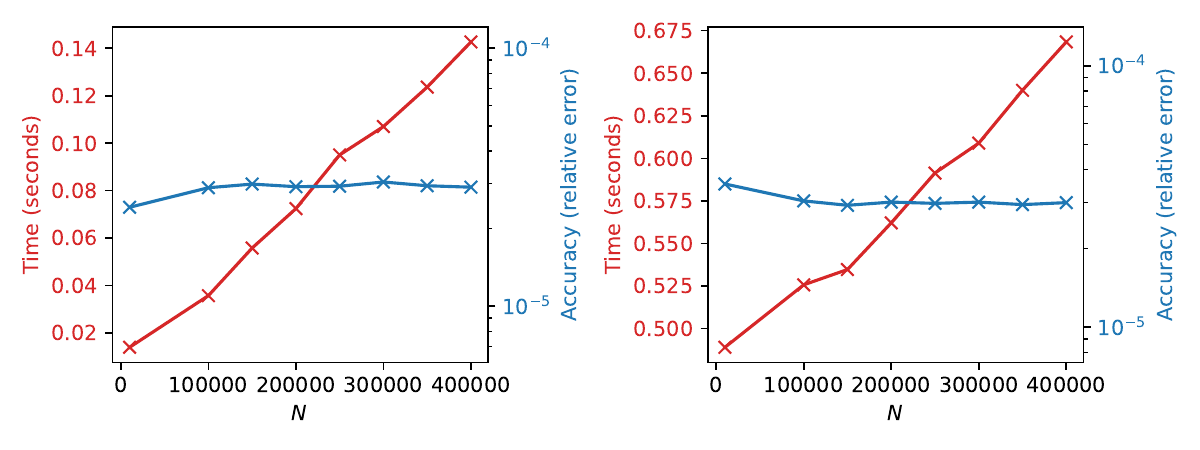}\caption{\label{fig:matvecN-matern}Matvec time and accuracy with varying $N$
in $d=1$ (left) and $d=2$ (right) for the Mat\'ern kernel with $\nu=\frac{3}{2}$.
The reference values for $N_{t}$, $N_{\sigma}$, and $M$ are the
same as in Figures \ref{fig:matvec1d} and \ref{fig:matvec2d}, respectively, \REV{as are the numbers of trial runs.}}
\end{figure}

Since the relative error is empirically independent of $N$, compatible
with Theorem \ref{thm:main}, we recommend an iterative approach to
choosing parameters given a required error: repeatedly run test matvecs
and increase one parameter, switching to another upon stagnation.
This procedure can be employed for small $N$ and then the choices
can be extrapolated to larger regression problems. 

Next we consider the same experiments in the case of the squared-exponential
kernel, for which it is unnecessary to choose $t_{\min}$ and $t_{\max}$,
and there is no need to vary $N_{t}$. In this setting, $\rho_{\min}=\sigma_{\min}$,
$\rho_{\max}=\sigma_{\max}$, and $\Delta\omega$ is determined accordingly.
The results of the same ablation experiments for $N_{\sigma}$ and
$M$ are presented in Figures \ref{fig:matvec1d-gauss}, \ref{fig:matvec2d-gauss},
and \ref{fig:matvec3d-gauss} for dimensions $d=1$, $2$, and $3$,
respectively. Once again, we take $N=10^{4}$ throughout.

\begin{figure}
\centering{}\includegraphics[bb=0bp 10bp 605bp 216bp,clip,scale=0.5]{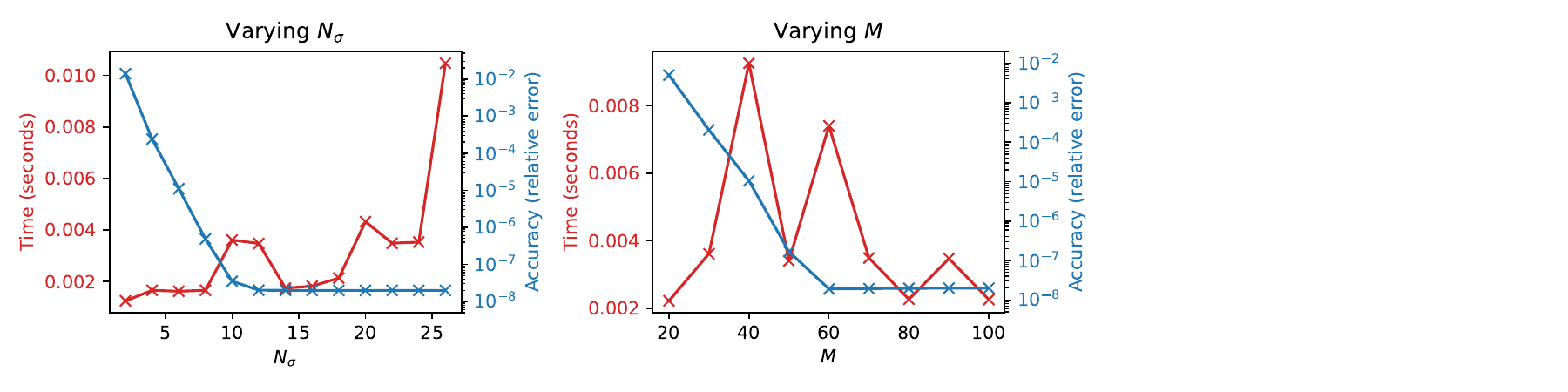}\caption{\label{fig:matvec1d-gauss}Matvec time and accuracy in one dimension
for the squared-exponential kernel. Unless otherwise specified, $N_{\sigma}=26$
and $M=100$. \REV{Results are averaged over 5 runs.}}
\end{figure}

\begin{figure}
\centering{}\includegraphics[bb=0bp 10bp 864bp 216bp,clip,scale=0.5]{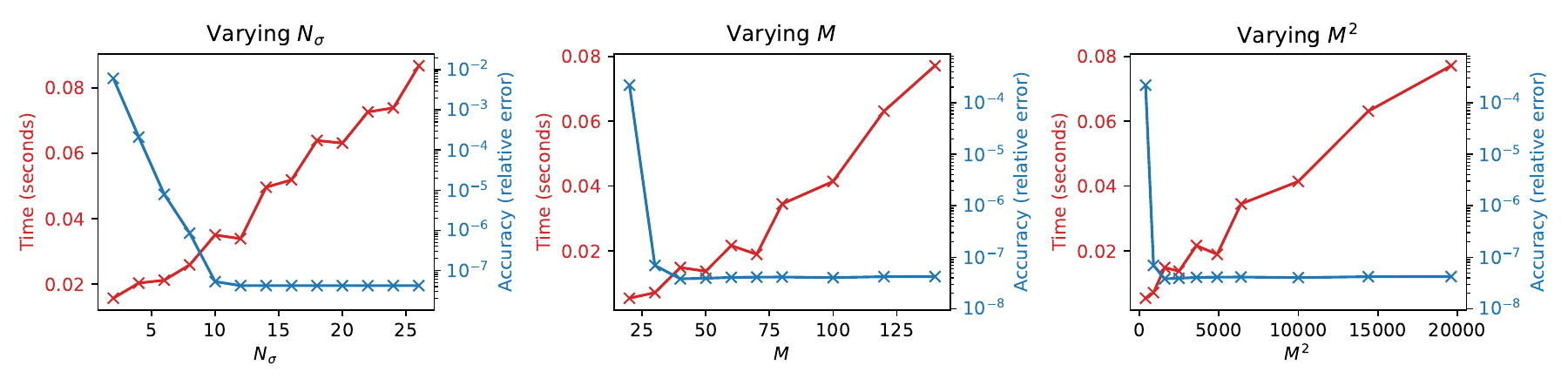}\caption{\label{fig:matvec2d-gauss}Matvec time and accuracy in two dimensions
for the squared-exponential kernel. Unless otherwise specified, $N_{\sigma}=26$
and $M=140$. \REV{Results are averaged over 3 runs.}}
\end{figure}

\begin{figure}
\begin{centering}
\includegraphics[bb=0bp 10bp 864bp 216bp,clip,scale=0.5]{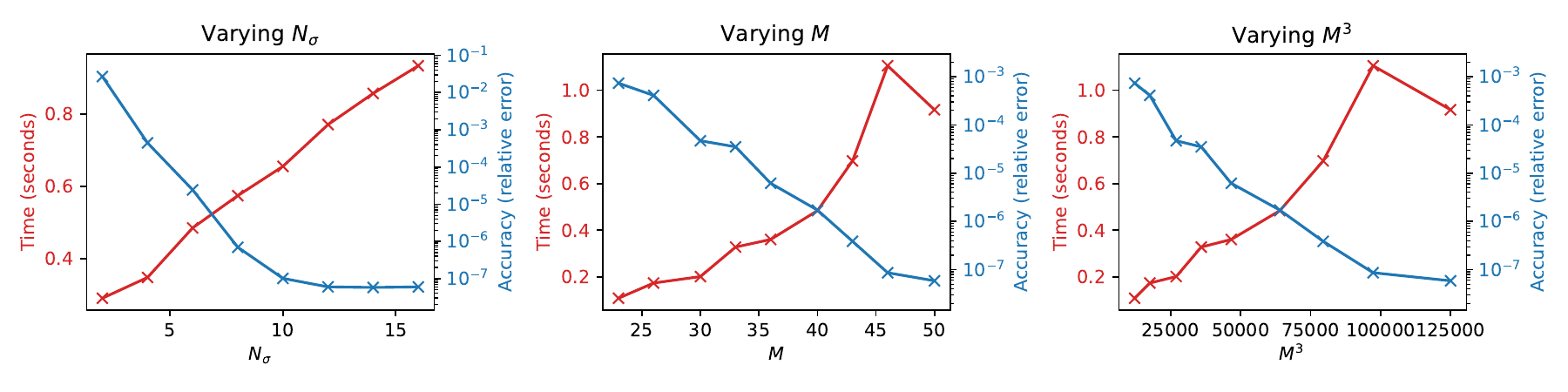}
\par\end{centering}
\caption{\label{fig:matvec3d-gauss}Matvec time and accuracy in three dimensions
for the squared-exponential kernel. Unless otherwise specified, $N_{\sigma}=16$
and $M=50$. \REV{Results are shown for a single run.}}
\end{figure}

The time scaling in Figure \ref{fig:matvec1d-gauss} is not obvious
because the error already saturates when $M$ and $N_{\sigma}$ are
both very small. Increasing the parameters to a size where the scaling
is more readily apparent would not be useful.

Once again, as expected, the error decays exponentially in $N_{\sigma}$
and $M$, and the time grows linearly in $N_{\sigma}$ and $M^{d}$.
Since the integration in $t$ no longer contributes to the error and
since $\rho_{\min}$ is larger (due to the fact that simply $\chi_{\min}=1$),
in this setting the error is smaller and the convergence with respect
to $M$ is faster, and we quickly hit the lower bound determined by
the tolerance $\epsilon$.

Finally, again keeping all parameters at their given reference values
(defined in Figures \ref{fig:matvec1d-gauss}, \ref{fig:matvec2d-gauss},
and \ref{fig:matvec3d-gauss}) and varying $N$, Figure \ref{fig:matvecN-gauss}
shows approximately linear scaling and constant relative error as
in the Mat\'ern case. For these experiments, we maintain relative accuracy
of roughly $10^{-7}$ independent of $N$.

\begin{figure}
\begin{centering}
\includegraphics[scale=0.5]{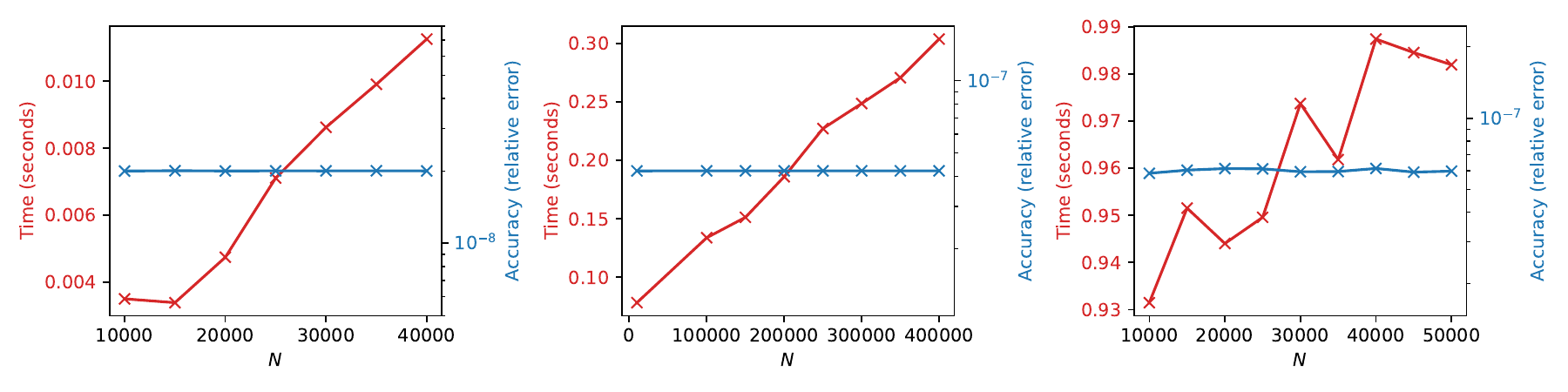}\caption{\label{fig:matvecN-gauss}Matvec time and accuracy with varying $N$
for the squared-exponential kernel, $d=1,2,3$. The reference values
for $N_{\sigma}$ and $M$ are the same as in Figures \ref{fig:matvec1d-gauss},
\ref{fig:matvec2d-gauss}, and \ref{fig:matvec3d-gauss}, \REV{as are the numbers of trial runs.}}
\par\end{centering}
\end{figure}

\subsection{Solves \label{sec:solves}}

For the purposes of GPR, we are specifically interested in wrapping
our kernel matvec subroutine in CG to compute the linear solve \REV{$(\mathbf{K}+\eta^{2}I_N ) \alpha = y$},
cf. \eqref{eq:gprsolve}. Since the solve cost is fairly insensitive
to the specification of \REV{$y$}, we simply draw the entries of \REV{$y \in\R^{N}$}
independently from the uniform distribution on $[0,1]$.

We compare our results for this inversion to those of the MATLAB package
FLAM \cite{ho_flam_2020}, which was the fastest out of all alternative
strategies tested in \cite{greengard_equispaced_2023}. In this section,
we focus on the squared-exponential kernel. We set $d=2$, since for
$d=1$ the scaling of FLAM is linear in $N$ and we believe that it
remains a state-of-the-art approach.

FLAM and our method take different approaches. FLAM first computes
a block factorization of the entire matrix $\mathbf{K}+\eta^{2}I_N$
using elementwise queries. This is expensive but allows for fast matvecs
and solves downstream. The Fourier method has insignificant offline
cost but more expensive matvecs, and the solves require an iterative
method wrapping the matvec subroutine. In order to compare the two,
we consider the sum of both the on- and offline costs, i.e., the time
necessary to go from no representation of $\mathbf{K}$ at all to
a complete solve \REV{of $(\mathbf{K}+\eta^{2}I_N) \alpha = y$}.

We consider two regimes, in both of which the condition number of
the kernel matrix should remain approximately constant as $N$ grows. \REV{Since our method uses an iterative solver, an increase in the condition number will increase its computational cost, while FLAM will remain relatively unaffected. Our method will be faster in the limit as $N \rightarrow \infty$ for any fixed condition number, but a change in conditioning could affect the crossover point.  In the ill-conditioned case, we could also use a less-accurate and therefore cheaper FLAM factorization to provide a preconditioner in our approach, though we do not pursue this technique here.}

In the first \REV{regime (the ``noisy regime'')}, we take $\eta^{2}=\Omega(N)$, specifically, $\eta^{2}=\frac{N}{2\times10^{6}}$,
while maintaining a constant kernel. In this scaling regime (the ``noisy regime''), although
we take an increasing number of observations, the noise of the observations
also grows, such that our uncertainty about the inferred function
remains constant. We fix $N_{\sigma}=15$, $M=75$, and $\epsilon=10^{-6}$,
and we choose a residual tolerance of $10^{-6}$ in CG. These choices
yield relative matvec errors that are never larger than $10^{-5}$.

In the second scaling regime \REV{(the ``narrow regime'')}, we keep $\eta^{2}=10^{-1}$ constant,
but change the kernel as $N$ increases so that the effective number
of nonzero entries in each row of $\mathbf{K}$ remains constant and
the diagonal entries of $\mathbf{K}$ remain $\Omega(1)$. Specifically,
in terms of our reference choice $\sigma_{\mathrm{ref}}(x)$ for $\sigma(x)$,
\REV{we set $\sigma(x)=(N/1000)^{-1/d}\sigma_{\mathrm{ref}}(x)$ and $w(x)=(N/1000)^{-1/2}$} to fix the kernel. We set $N_{\sigma}=15$, independent of $N$, and
$M=3\,N^{1/d}$ in order to maintain constant relative accuracy as
$N$ is increased. These choices yield relative matvec errors that
are never larger than $10^{-5}$. Again we choose a residual tolerance
of $10^{-6}$ in CG. \REV{FLAM uses a recursive skeletonization algorithm, based on multiple interpolative decompositions (IDs) of submatrices.  We set the tolerance for each ID to $10^{-10}$ and choose a {tree occupancy parameter}, which governs the size of the smallest submatrices to be factorized, of $200$. The exact choice here is not particularly important in our experiments as long as it is substantially less than the size of the matrix. Both of these choices follow \cite{greengard_equispaced_2023}.}

\begin{figure}
\begin{centering}
\includegraphics[scale=0.5]{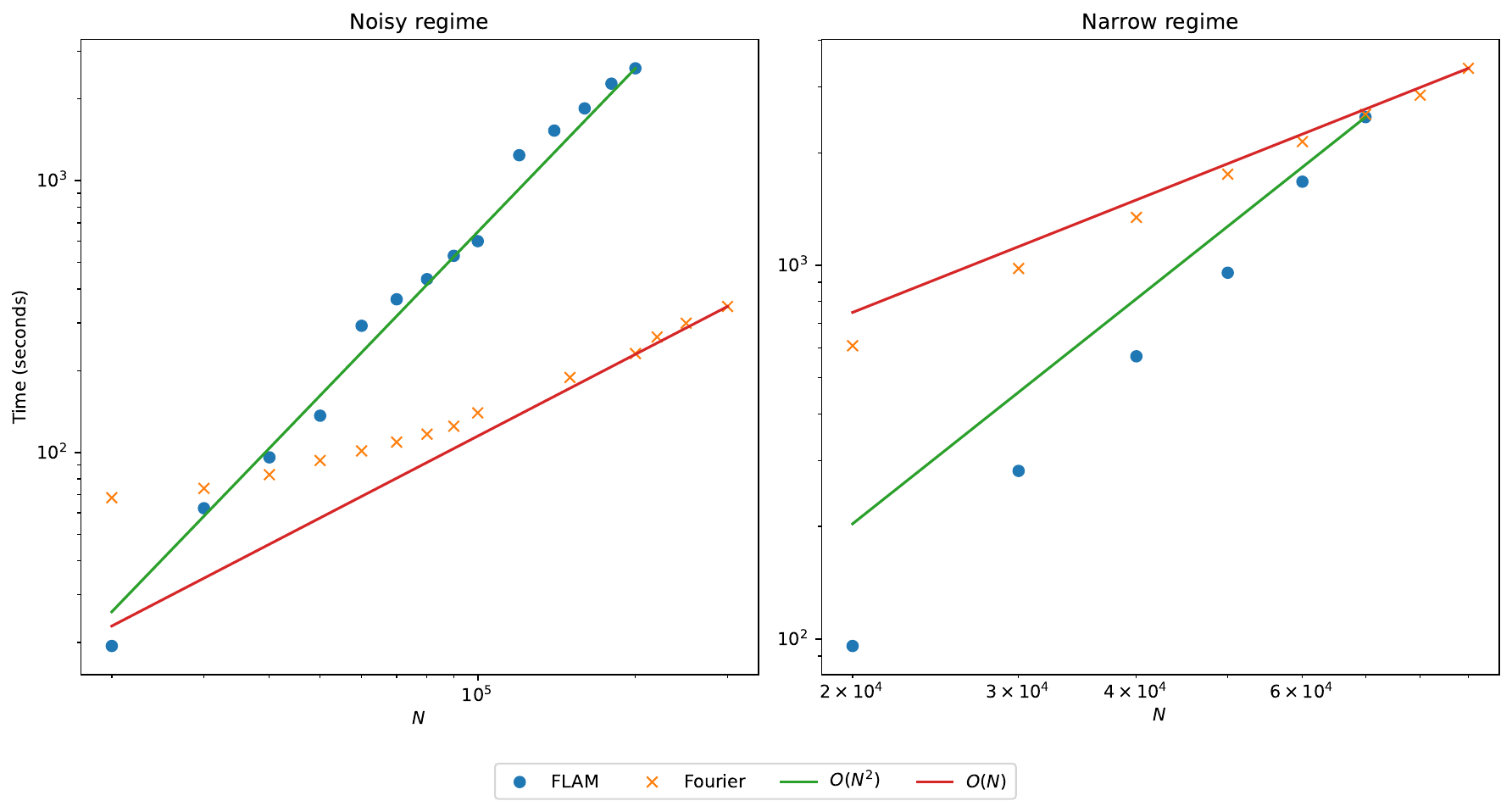}
\par\end{centering}
\centering{}\caption{\label{fig:inversion}Time to perform the linear solve $(\mathbf{K}+\eta^{2}I_N ) \alpha = y$
for the squared exponential kernel in $d=2$ in the scaling regimes discussed in Section \ref{sec:solves}.}
\end{figure}

\begin{figure}
\begin{centering}
\includegraphics[scale=0.5]{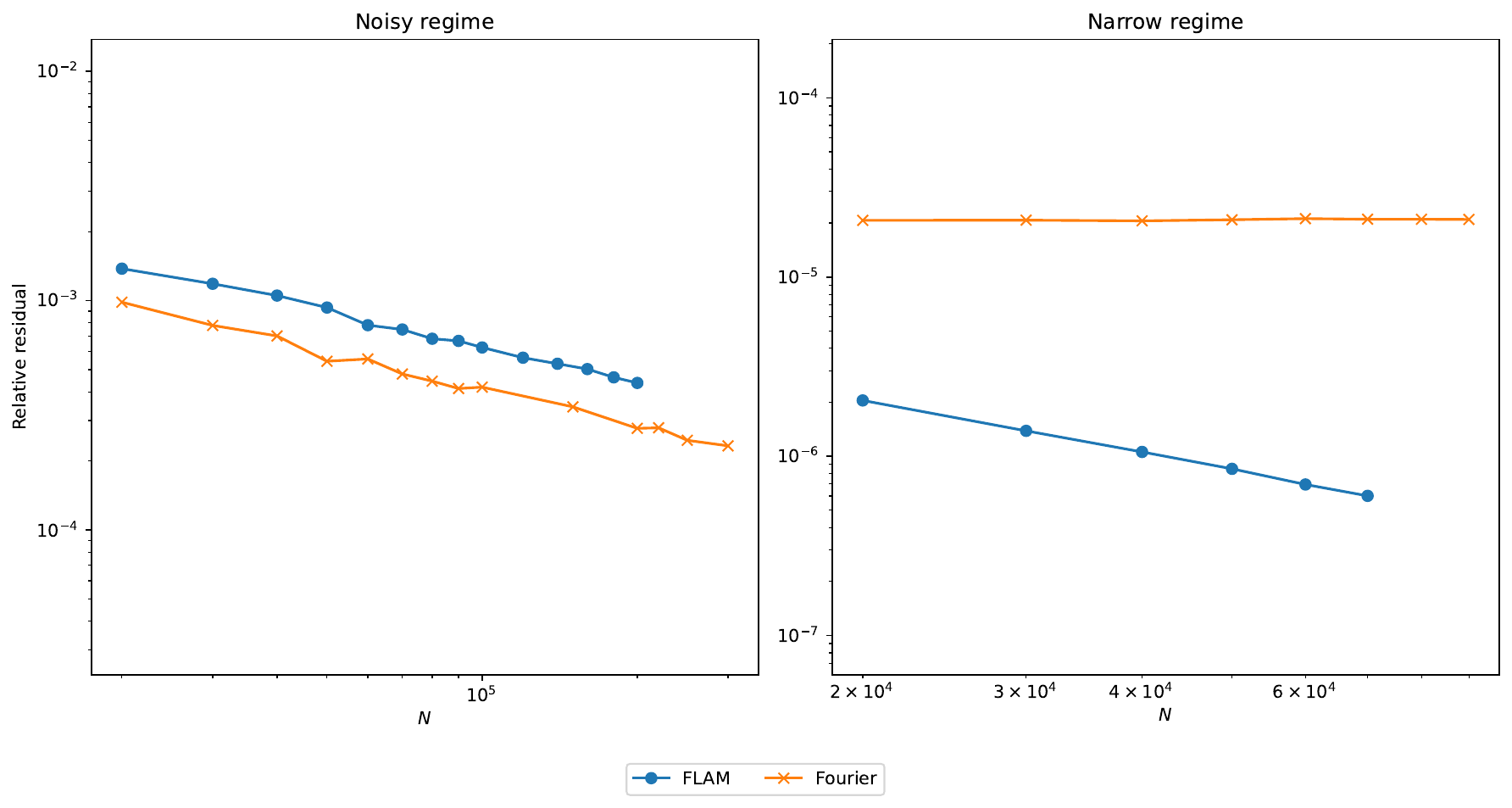}
\par\end{centering}
\centering{}\caption{\label{fig:inversion_residuals}Relative residuals from performing the solve $(\mathbf{K}+\eta^{2}I_N) \alpha = y $
for the squared exponential kernel in $d=2$ in the scaling regimes discussed in Section \ref{sec:solves}.}
\end{figure}

\REV{We report the cost comparison of the two approaches in Figure \ref{fig:inversion}, and we show the relative residuals $\Vert (\mathbf{K}+\eta^2 I_N) \alpha - y \Vert/\Vert y \Vert$ obtained by each method in Figure \ref{fig:inversion_residuals}, computed using the exact\footnote{The exact applications of $\mathbf{K}$ required to compute the residuals were carried out in $O(N^2)$ time using the KeOps library \cite{keops}.  Computing the solves this way would have been prohibitively slow, but a single application required only a few minutes for the largest values of $N$ that we considered.} kernel matrix $\mathbf{K}$. These figures confirm the $\tilde{O}(N)$ scaling of our method, which provides comparable residuals to those of FLAM, while the scaling of FLAM is at least $O(N^2)$ in both cases.  (We use a naive implementation, in which we simply provide the FLAM library with a function for element- or block-wise access to the kernel matrix; more GPR-specific implementations using annular proxies can achieve $O(N^{3/2})$ scaling, which is still worse than the demonstrated scaling of our method \cite{greengard_equispaced_2023}.)  The narrow regime uses smaller values of $N$ compared to the noisy regime; this is due to both our implementation and FLAM running out of memory on the consumer hardware we used.  For our method, this is caused by the scaling of $M^d =O(N)$ Fourier grid points in this regime, exacerbating the $O(M^{d} N_{\sigma})$ memory bottleneck pointed out earlier in Section~\ref{sec:matvecs}. The same memory restrictions kept us from experimenting with larger values of $M$ and $N_\sigma$, which we would expect to bring the residuals in the narrow regime in line with those of FLAM. This memory bottleneck is straightforward to alleviate with a slightly modified implementation, as we have already commented in Section~\ref{sec:matvecs}, in which we calculate the individual pointwise multiplications in Fourier space sequentially or in smaller batches. We defer any such alternative implementations, as well as analysis of the tradeoff between speed and memory, for future application-oriented work.}

\bibliographystyle{plain}
\bibliography{gpr_nonstationary}

\appendix

\part*{Appendices}

\section{Glossary of notation\label{app:glossary}}

The following is a glossary of commonly used notation in our analysis.
We think of $d$, $\nu$, $\sigma_{\min}$, and $\sigma_{\max}$ as
given. We think of $t_{\min},t_{\max},N_{t}$, as well as $N_{\sigma}$
and $M,\Delta\omega$ as tunable parameters controlling the error
of the approximation. The other quantities are all induced in terms
of these.
\begin{center}
\begin{tabular}{|c|c|}
\hline 
\textbf{Terms} & \textbf{Meaning}\tabularnewline
\hline 
\hline 
$d$ & Spatial dimension\tabularnewline
\hline 
$\nu$ & Mat\'ern parameter (\ref{eq:matern})\tabularnewline
\hline 
\multirow{2}{*}{$t_{\min}$, $t_{\max}$, $N_{t}$} & \multicolumn{1}{c|}{Parameters for discretization scheme (\ref{eq:dtscheme}) of integral }\tabularnewline
 & representation (cf. (\ref{eq:schoenberg2}) and (\ref{eq:maternpres}))
of the Mat\'ern kernel \tabularnewline
\hline 
\multirow{2}{*}{$\sigma_{\min}$, $\sigma_{\max}$, $N_{\sigma}$, $\kappa$} & Parameters for Chebyshev interpolation by value of \tabularnewline
 & $\sigma(x)\in[\sigma_{\min},\sigma_{\max}]$, cf. (\ref{eq:chebscheme});
$\kappa=\sigma_{\max}/\sigma_{\min}$.\tabularnewline
\hline 
$\Lambda_{N_{\sigma}}$ & Lebesgue constant for the Chebyshev interpolation scheme, cf. Definition
\ref{def:lebesgue}\tabularnewline
\hline 
$\chi_{\min}$, $\chi_{\max}$ & $\chi_{\min}=\chi(t_{\min})$ and $\chi_{\max}=\chi(t_{\max})$, where
$\chi(t)$ is defined in (\ref{eq:maternpres})\tabularnewline
\hline 
$\rho_{\min}$, $\rho_{\max}$, $\lambda$ & $\rho_{\min}=\sigma_{\min}\chi_{\min}$, $\rho_{\max}=\sigma_{\max}\chi_{\max}$,
and $\lambda=\rho_{\max}/\rho_{\min}$\tabularnewline
\hline 
$M$, $\Delta\omega$ & Parameters defining Fourier grid $\Delta\omega\times\{-M,\ldots,M\}^{d}$,
cf. Section \ref{sec:approxgrid}\tabularnewline
\hline 
\end{tabular}
\par\end{center}

\section{Integral representation of the kernel \label{app:pres}}

Here we verify (\ref{eq:pres}). Recalling the definition of $\mathcal{B}_{t}$,
we compute: 
\begin{align*}
\left(\mathcal{B}_{t}\mathcal{B}_{t}^{*}\right)(x,y) & =\int\mathcal{B}_{t}(x,z)\mathcal{B}_{t}(y,z)\,dz\\
 & =w(x)\,w(y)\left[2\pi\,\sigma(x)\sigma(y)\right]^{-d}\int e^{-\frac{\vert x-z\vert^{2}}{2\sigma^{2}(x)\chi^{2}(t)}}e^{-\frac{\vert y-z\vert^{2}}{2\sigma^{2}(y)\chi^{2}(t)}}\,dz.
\end{align*}
 We can substitute the exact integration
\[
\int e^{-\frac{\vert x-z\vert^{2}}{2\sigma(x)^{2}\chi(t)^{2}}}e^{-\frac{\vert y-z\vert^{2}}{2\sigma(y)^{2}\chi(t)^{2}}}\,dz=\left(\frac{2\pi\sigma^{2}(x)\sigma^{2}(y)\chi^{2}(t)}{\sigma^{2}(x)+\sigma^{2}(y)}\right)^{d/2}e^{-\frac{\vert x-y\vert^{2}}{2\left(\sigma^{2}(x)+\sigma^{2}(y)\right)\chi^{2}(t)}}
\]
 to obtain 
\begin{equation}
\left(\mathcal{B}_{t}\mathcal{B}_{t}^{*}\right)(x,y)\,v(t)=w(x)\,w(y)\left(2\pi\left[\sigma^{2}(x)+\sigma^{2}(y)\right]\right)^{-d/2}\,e^{-\frac{\vert x-y\vert^{2}}{2\left(\sigma^{2}(x)+\sigma^{2}(y)\right)\chi^{2}(t)}}u(t)\label{eq:BtBtstar}
\end{equation}

Then using (\ref{eq:schoenberg2}) to integrate against $u(t)\,dt$,
we obtain 
\begin{align*}
\int\left(\mathcal{B}_{t}\mathcal{B}_{t}^{*}\right)(x,y)\,v(t)\,dt & =w(x)\,w(y)\left(2\pi\left[\sigma^{2}(x)+\sigma^{2}(y)\right]\right)^{-d/2}\int e^{-\frac{\vert x-y\vert^{2}}{2\left(\sigma^{2}(x)+\sigma^{2}(y)\right)\chi^{2}(t)}}\,u(t)\,dt\\
 & =w(x)\,w(y)\left(2\pi\left[\sigma^{2}(x)+\sigma^{2}(y)\right]\right)^{-d/2}\vp\left(\frac{\vert x-y\vert}{\sqrt{\sigma^{2}(x)+\sigma^{2}(y)}}\right)\\
 & =\mathcal{K}(x,y).
\end{align*}
 This establishes the desired identity (\ref{eq:pres}).

\section{Schoenberg representation of the Mat\'ern function \label{app:matern}}

We want to show how to present such $\vp=\vp_{\nu}$ in the form of
(\ref{eq:schoenberg2}) for suitably chosen $u(t)$ and $\chi(t)$. 

It is useful to recall \cite{stein_interpolation_1999} that the Mat\'ern
function (\ref{eq:matern}) can be more conveniently characterized
in terms of a Fourier transform. Specifically, one can compute that 

\begin{equation}
\psi(r):=\frac{\sqrt{\pi}}{2^{\nu-1}\Gamma\left(\nu+\frac{1}{2}\right)}\vert r\vert^{\nu}K_{\nu}(r)=\int\frac{1}{(1+\omega^{2})^{\beta}}\,e^{i\omega r}\,d\omega,\label{eq:maternfourier}
\end{equation}
 where $\beta=\nu+\frac{1}{2}$. In terms of $\psi$, we can write
the Mat\'ern function (\ref{eq:matern}) $\vp=\vp_{\nu}$ as 
\begin{equation}
\vp(r)=\frac{\Gamma(\nu+\frac{1}{2})}{\sqrt{\pi}\,\Gamma(\nu)}\psi(\sqrt{2\nu}\,r).\label{eq:relate}
\end{equation}

Now recall the identity \cite{beylkin_approximation_2005}: 
\[
\frac{1}{(1+\omega^{2})^{\beta}}=\frac{1}{\Gamma(\beta)}\int e^{-e^{t}(1+\omega^{2})+\beta t}\,dt.
\]
 Then substitute into (\ref{eq:maternfourier}) to compute: 

\begin{align}
\psi(r) & =\int\frac{1}{(1+\omega^{2})^{p}}\,e^{i\omega r}\,d\omega\nonumber \\
 & =\frac{1}{\Gamma(\beta)}\int\int e^{-e^{t}(1+\omega^{2})+\beta t}\,e^{i\omega r}\,d\omega\,dt\nonumber \\
 & =\frac{1}{\Gamma(\beta)}\int e^{\beta t-e^{t}}\left[\int e^{-e^{t}\omega^{2}}\,e^{i\omega r}\,d\omega\right]dt\nonumber \\
 & \overset{(\star)}{=}\frac{1}{\Gamma(\beta)}\int e^{\beta t-e^{t}}\sqrt{\frac{\pi}{e^{t}}}e^{-\frac{r^{2}}{4e^{t}}}\,dt\nonumber \\
 & =\frac{\sqrt{\pi}}{\Gamma(\beta)}\int e^{(\beta-\frac{1}{2})t-e^{t}}e^{-\frac{r^{2}}{4e^{t}}}\,dt.\label{eq:psicomp}
\end{align}
 In the step indicated by $(\star)$ we have used the identity for
the inverse Fourier transform of a Gaussian.

Then plugging (\ref{eq:psicomp}) into (\ref{eq:relate}) and recalling
that $\beta=\nu+\frac{1}{2}$, we obtain
\[
\vp(r)=\frac{1}{\Gamma(\nu)}\int e^{\nu t-e^{t}}e^{-\frac{r^{2}}{2e^{t}/\nu}}\,dt,
\]
 from which it directly follows that (\ref{eq:schoenberg2}) holds
with 
\[
u(t):=\frac{1}{\Gamma(\nu)}e^{\nu t-e^{t}},\quad\chi(t):=\frac{1}{\sqrt{\nu}}e^{t/2}.
\]

\section{Background on NUFFTs \label{app:nuffts}}

In our algorithms we will make use of \emph{non-uniform fast Fourier
transforms} (\emph{NUFFTs}) \cite{BOYD1992243,DuttRokhlin} of both
type 1 and type 2, which are in fact formal adjoints of one another.

For our purposes, the \textbf{type 1 NUFFT} takes as input a vector
of values $(\alpha_{j})_{j=1}^{N}$ associated to a scattered set
of points $x_{1},\ldots,x_{N}\in\R^{d}$ and returns the object $\hat{\alpha}[\mathbf{n}]$
defined on elements $\mathbf{n}\in\mathbf{G}:=\{-M,\ldots,M\}^{d}$
of a truncated integer lattice by the formula $\hat{\alpha}[\mathbf{n}]=\sum_{j=1}^{N}\alpha_{j}e^{-2\pi i\,\mathbf{n}\cdot x_{j}}.$
By rescaling the scattered points $x_{j}\rightarrow(\Delta\omega)\,x_{j}$,
we can define the rescaled type-1 NUFFT: 
\[
\hat{\alpha}[\mathbf{n}]=\sum_{j=1}^{N}\alpha_{j}e^{-2\pi i\,(\mathbf{n}\Delta\omega)\cdot x_{j}},
\]
 which can be viewed as the evaluation of the unitary Fourier transform
$\mathcal{F}$ of the empirical distribution $\sum_{j=1}^{N}\alpha_{j}\delta_{x_{j}}$
on the Fourier grid $\left\{ \mathbf{n}\Delta\omega\,:\,\mathbf{n}\in\mathbf{G}\right\} $:
\[
\hat{\alpha}[\mathbf{n}]=\left(\mathcal{F}\left[\sum_{j=1}^{N}\alpha_{j}\delta_{x_{j}}\right]\right)(\mathbf{n}\Delta\omega).
\]

The \textbf{type 2 NUFFT} (including our rescaling convention), takes
such a grid function $f$ as input and returns a vector $\hat{f}=(\hat{f}_{j})_{j=1}^{N}$
of values associated to the scattered points, defined by: 
\[
\hat{f}_{j}=\sum_{\mathbf{n}\in\mathbf{G}}f[\mathbf{n}]\,e^{2\pi i(\mathbf{n}\Delta\omega)\cdot x_{j}}.
\]
 Importantly, although the NUFFTs of type 1 and 2 are formal adjoints
of one another, they are not inverses of one another.

Both types of NUFFT can be computed approximately (with prescribed
accuracy) using standard libraries with $O(M^{d}+N\log N)$ computational
cost \cite{barnett_parallel_2019}. All of our computations make use
of the FiNUFFT library \cite{barnett_parallel_2019,cufinufft,barnett_aliasing_2021}.
In this work, we assume for simplicity that we can compute NUFFTs
exactly, neglecting this source of approximation error. We comment
here that the scaling of the runtime of the NUFFT with respect to
the prescribed error is merely logarithmic and can be easily accounted
for. We elide this source of error simply for clarity of presentation.

\section{Proofs of technical lemmata \label{app:lemmas}}
\begin{proof}
[Proof of Lemma~\ref{lem:holder}] Without loss of generality assume
that $\trip A\trip=1$. Then letting $x\in\R^{N}$ with $\Vert x\Vert_{p}=1$,
we want to show that $\Vert Ax\Vert_{p}\leq N$. Let $a_{i}$ be the
$i$-th row of $A$, and let $q$ be such that $1/p+1/q=1$. Then
$\Vert a_{i}\Vert_{q}\leq N^{1/q}$ follows from $\trip A\trip=1$
for all $i=1,\ldots,N$. Finally, we can use H\"older's inequality to
bound 
\[
\Vert Ax\Vert_{p}^{p}=\sum_{i=1}^{N}\left|a_{i}\cdot x\right|^{p}\leq\sum_{i=1}^{N}\Vert a_{i}\Vert_{q}^{p}\Vert x\Vert_{p}^{p}\leq N^{p/q+1}=N^{p},
\]
 which completes the proof.
\end{proof}
\vspace{2mm} \noindent \begin{center} \rule{0.5\columnwidth}{0.5pt}
\end{center} \vspace{2mm}
\begin{proof}
[Proof of Lemma \ref{lem:dt}] Recalling our expression (\ref{eq:BtBtstar0})
for $\left(\mathcal{B}_{t}\mathcal{B}_{t}^{*}\right)(x,y)\,v(t)$,
our assumption (\ref{eq:w_assumption}) bounding $\vert w(x)\vert\leq1$,
and our Mat\'ern function representation (\ref{eq:maternpres}), compute:
\begin{align*}
\left|\left(\mathcal{B}_{t}\mathcal{B}_{t}^{*}\right)(x,y)\,v(t)\right| & =\ \left|w(x)\,w(y)\left(2\pi\left[\sigma^{2}(x)+\sigma^{2}(y)\right]\right)^{-d/2}\,e^{-\frac{\vert x-y\vert^{2}}{2\left(\sigma(x)^{2}+\sigma(y)^{2}\right)\chi(t)^{2}}}\,u(t)\right|\\
 & \leq\ (2\pi)^{-d/2}\sigma_{\min}^{-d}\,u(t)\\
 & =\ O\left(\sigma_{\min}^{-d}e^{\nu t-e^{t}}\right).
\end{align*}
 In summary, by Definition \ref{def:triplenorm} of the norm $\lrtrip{\cdot}$,
\begin{equation}
\lrtrip{\mathcal{B}_{t}\mathcal{B}_{t}^{*}v(t)}=O\left(\sigma_{\min}^{-d}e^{\nu t-e^{t}}\right).\label{eq:btvpointwise}
\end{equation}

Now, extending our definition $t_{j}=t_{\min}+j\,\Delta t$ to all
integers $j\in\mathbb{Z}$ (where we recall that $\Delta t=(t_{\max}-t_{\min})/N_{t}$),
we bound by the triangle inequality: 
\begin{align}
\lrtrip{\mathcal{K}-\sum_{j=0}^{N_{t}}(\mathcal{B}_{t_{j}}\mathcal{B}_{t_{j}}^{*})\,v_{j}} & \ \leq\ \lrtrip{\mathcal{K}-\sum_{j=-\infty}^{\infty}(\mathcal{B}_{t_{j}}\mathcal{B}_{t_{j}}^{*})\,v(t_{j})\,\Delta t}\nonumber \\
 & \quad\quad+\ \sum_{j=-\infty}^{0}\lrtrip{\mathcal{B}_{t_{j}}\mathcal{B}_{t_{j}}^{*}\,v(t_{j})}\,\Delta t+\sum_{j=N_{t}}^{\infty}\lrtrip{\mathcal{B}_{t_{j}}\mathcal{B}_{t_{j}}^{*}\,v(t_{j})}\,\Delta t.\nonumber \\
 & \ =\ \lrtrip{\mathcal{K}-\sum_{j=-\infty}^{\infty}(\mathcal{B}_{t_{j}}\mathcal{B}_{t_{j}}^{*})\,v(t_{j})\,\Delta t}\nonumber \\
 & \quad\quad+\ O\left(\sigma_{\min}^{-d}\sum_{j=-\infty}^{0}e^{\nu t_{j}-e^{t_{j}}}\,\Delta t+\sigma_{\min}^{-d}\sum_{j=N_{t}}^{\infty}e^{\nu t_{j}-e^{t_{j}}}\,\Delta t\right),\label{eq:Kriembound}
\end{align}
 where in the last step we have used (\ref{eq:btvpointwise}). In
the last expression, the first error term measures the error of the
infinite Riemann sum approximation of $\mathcal{K}=\int_{-\infty}^{\infty}\mathcal{B}_{t}\mathcal{B}_{t}^{*}\,v(t)\,dt$,
while the last term measures the effect of cutting off the tails of
the Riemann sum.

The impact of the tails can be bounded as follows. First compute:
\[
\sum_{j=-\infty}^{0}e^{\nu t_{j}-e^{t_{j}}}\,\Delta t\leq\Delta t\sum_{j=-\infty}^{0}e^{\nu t_{j}}=e^{\nu t_{\min}}(\Delta t)\sum_{j=0}^{\infty}e^{-(\nu\Delta t)j}=e^{\nu t_{\min}}\frac{\Delta t}{1-e^{-\nu\Delta t}},
\]
 so it follows that 
\begin{equation}
\sum_{j=-\infty}^{0}e^{\nu t_{j}-e^{t_{j}}}\,\Delta t=O(e^{\nu t_{\min}}),\label{eq:tail1}
\end{equation}
 as long as $\Delta t=O(1)$.

Next, note that by convexity, $e^{t}\geq e^{t_{0}}+(t-t_{0})e^{t_{0}}$
for any $t_{0}$. Choose $t_{0}=\log(\nu+\alpha)$ for arbitrary $\alpha>0$,
yielding 
\[
e^{t}\geq(\nu+\alpha)t+(\nu+\alpha)(1-\log(\nu+\alpha)).
\]
 Therefore 
\[
\sum_{j=N_{t}}^{\infty}e^{\nu t_{j}-e^{t_{j}}}\leq e^{-(\nu+\alpha)(1-\log(\nu+\alpha))}\sum_{j=N_{t}}^{\infty}e^{-\alpha t_{j}}\,dt=e^{-(\nu+\alpha)(1-\log(\nu+\alpha))}\frac{e^{-\alpha t_{\max}}}{1-e^{-\nu\Delta t}},
\]
 and it follows that 
\begin{equation}
\sum_{j=N_{t}}^{\infty}e^{\nu t_{j}-e^{t_{j}}}\,\Delta t\leq e^{-(\nu+\alpha)(1-\log(\nu+\alpha))}e^{-\alpha t_{\max}}\frac{\Delta t}{1-e^{-\nu\Delta t}}=O(e^{-\alpha t_{\max}}).\label{eq:tail2}
\end{equation}

Then combining (\ref{eq:tail1}) and (\ref{eq:tail2}) with (\ref{eq:Kriembound}),
we have that for any $\alpha>0$, 
\begin{equation}
\lrtrip{\mathcal{K}-\sum_{j=0}^{N_{t}}(\mathcal{B}_{t_{j}}\mathcal{B}_{t_{j}}^{*})\,v_{j}}\leq\lrtrip{\mathcal{K}-\sum_{j=-\infty}^{\infty}(\mathcal{B}_{t_{j}}\mathcal{B}_{t_{j}}^{*})\,v(t_{j})\,\Delta t}+O\left(\sigma_{\min}^{-d}\,[e^{-\alpha t_{\max}}+e^{\nu t_{\min}}]\right).\label{eq:Kriembound2}
\end{equation}
 It remains to bound the first term on the right-hand side of (\ref{eq:Kriembound2}),
i.e., the error of the infinite Riemann sum.

Specifically, we want to show that for any $\delta\in(0,1)$, 
\begin{equation}
\lrtrip{\mathcal{K}-\sum_{j=-\infty}^{\infty}(\mathcal{B}_{t_{j}}\mathcal{B}_{t_{j}}^{*})\,v(t_{j})\,\Delta t}=O\left(\sigma_{\min}^{-d}e^{-\frac{(1-\delta)\pi^{2}}{\Delta t}}\right),\label{eq:trapgoal}
\end{equation}
 which, together with (\ref{eq:Kriembound2}), will complete the proof.

For simplicity, fix $x,y$ and define 
\begin{equation}
f(t):=e^{-\frac{\vert x-y\vert^{2}}{2\left(\sigma(x)^{2}+\sigma(y)^{2}\right)\chi(t)^{2}}}\,e^{\nu t-e^{t}}.\label{eq:fdef}
\end{equation}
 Then, by our expression (\ref{eq:BtBtstar0}) for $\left(\mathcal{B}_{t}\mathcal{B}_{t}^{*}\right)(x,y)\,v(t)$,
our assumption (\ref{eq:w_assumption}) bounding $\vert w(x)\vert\leq1$,
and our Mat\'ern function representation (\ref{eq:maternpres}), it suffices
to show that the error of the infinite Riemann sum for $f(t)$ is
$O\left(e^{-\frac{(1-\delta)\pi^{2}}{\Delta t}}\right)$, independent
of $x,y$.

Standard results \cite{trefethen_traprule_2014} guarantee that if
we extend $f$ analytically to a strip 
\[
T\coloneqq\{z\in\mathbb{C}\,:\,|\Im(z)|<B\}
\]
 where $B>0$ and can find $M>0$ bounding 
\[
\int_{-\infty}^{\infty}|f(a+ib)|\,da\leq M
\]
 uniformly over $b\in(-B,B)$, then the error of the Riemann sum is
bounded by 
\[
\frac{2M}{e^{2\pi B/\Delta t}-1}.
\]
 We will take $B=(1-\delta)\frac{\pi}{2}$, so the proof is complete
once we can show that 
\begin{equation}
\int_{-\infty}^{\infty}|f(a+ib)|\,da=O(1),\label{eq:suffice}
\end{equation}
 independent of $x,y$ and $\vert b\vert\le(1-\delta)\frac{\pi}{2}$.

For $t=a+ib$ with $\vert b\vert\le(1-\delta)\frac{\pi}{2}$, let
us us bound the size of the first factor in the definition (\ref{eq:fdef})
of $f(t)$, using the fact that $\chi^{2}(t)=\nu^{-1}e^{t}$: 
\begin{equation}
\left|e^{-\frac{\vert x-y\vert^{2}}{2\nu^{-1}\left(\sigma(x)^{2}+\sigma(y)^{2}\right)e^{t}}}\right|=\left|e^{-e^{-t}}\right|^{\frac{\vert x-y\vert^{2}}{2\nu^{-1}\left(\sigma(x)^{2}+\sigma(y)^{2}\right)}}.\label{eq:complexgaussian}
\end{equation}
Now 
\[
\left|e^{-e^{-t}}\right|=e^{-\Re[e^{-t}]}=e^{-e^{-a}\cos(b)},
\]
 and since $\vert b\vert<\pi/2$, we have $e^{-a}\cos(b)>0$, and
therefore $\left|e^{-e^{-t}}\right|\leq1$. Moreover, the power $\frac{\vert x-y\vert^{2}}{2\nu^{-1}\left(\sigma(x)^{2}+\sigma(y)^{2}\right)}$
is non-negative, so the entirety of (\ref{eq:complexgaussian}) is
bounded by $1$ on $T$, i.e., 
\begin{equation}
\left|e^{-\frac{\vert x-y\vert^{2}}{2\nu^{-1}\left(\sigma(x)^{2}+\sigma(y)^{2}\right)e^{t}}}\right|\leq1\label{eq:exp1}
\end{equation}
 for $t=a+ib$ where $\vert b\vert\le(1-\delta)\frac{\pi}{2}$.

Next, let us us bound the size of the second factor in the definition
(\ref{eq:fdef}) of $f(t)$: 
\[
\left|e^{\nu t-e^{t}}\right|=e^{\Re[\nu t-e^{t}]}=e^{\nu a-e^{a}\cos(b)}.
\]
 Now $\vert b\vert<B:=(1-\delta)\frac{\pi}{2}$, so $\cos(b)\geq\cos(B)>0$,
so 
\begin{equation}
\left|e^{\nu t-e^{t}}\right|\leq e^{\nu a-\cos(B)e^{a}}.\label{eq:exp2}
\end{equation}

Therefore, combining (\ref{eq:exp1}) and (\ref{eq:exp2}) with the
definition (\ref{eq:fdef}) of $f(t)$, we obtain 

\[
\int_{-\infty}^{\infty}|f(a+ib)|\,da\leq\int_{-\infty}^{\infty}e^{\nu a-\cos(B)e^{a}}\,da.
\]
 The right-hand side is integrable since $\cos(B)>0$ and independent
of $x,y$, as well as the choice of $b\in(-B,B)$, so we have achieved
(\ref{eq:suffice}) and the proof is complete.
\end{proof}
\vspace{2mm} \noindent \begin{center} \rule{0.5\columnwidth}{0.5pt}
\end{center} \vspace{2mm}
\begin{proof}
[Proof of Lemma~\ref{lem:interp}] Our proof will leverage the well-known
error bound for Chebyshev interpolation given by Theorem 8.2 of \cite{trefethen_approximation_2020}.
To use this error bound, we make a few preliminary comments. First,
recall that the \emph{Bernstein ellipse }$E_{\rho}$ of radius $\rho>1$
is defined to be the subset 
\[
E_{\rho}=\left\{ \frac{\rho e^{i\theta}+\rho^{-1}e^{-i\theta}}{2}:\theta\in[0,2\pi)\right\} =\left\{ z\in\mathbb{C}\,:\,\frac{\Re(z)^{2}}{\left[\frac{\rho+\rho^{-1}}{2}\right]^{2}}+\frac{\Im(z)^{2}}{\left[\frac{\rho-\rho^{-1}}{2}\right]^{2}}\right\} 
\]
 of the complex plane.

Second, let $\sigma(s):=\sigma_{\min}+\frac{s+1}{2}\left(\sigma_{\max}-\sigma_{\min}\right)$
be the affine transformation that maps the standard reference interval
$[-1,1]$ for Chebyshev interpolation to our interpolation interval
$[\sigma_{\min},\sigma_{\max}]$. We can view this map $\sigma(s)$
as extending to the entire complex plane.

Now Theorem 8.2 of \cite{trefethen_approximation_2020} directly implies
that if there exists some constant $M$ such that $\left|e^{-\frac{x^{2}}{2\sigma(s)^{2}}}\right|\leq M$
for all $s \in E_{\rho}$ and all $x\in\R$, then 
\begin{equation}
\left|e^{-\frac{x^{2}}{2\sigma(t)^{2}}}-\sum_{k=1}^{N_{\sigma}}P_{k}(\sigma(t))e^{-\frac{x^{2}}{2\sigma_{k}^{2}}}\right|\leq\frac{4M}{\rho-1}\rho^{-N_{\sigma}}\label{eq:trefethen}
\end{equation}
 for all $\sigma(t) \in[\sigma_{\min},\sigma_{\max}]$ and all $x\in\R$.
Therefore, we want to find constants $\rho>1$ and $M$ such that
the bound 
\begin{equation}
\left|e^{-\frac{x^{2}}{2\sigma(s)^{2}}}\right|\leq M\label{eq:ellipsebound}
\end{equation}
 holds for all $s\in E_{\rho}$ and $x\in\R$.

To achieve this, we start by writing $s=a+bi$ for $a,b\in\R$. Then
\[
\sigma_{\min}^{-1}\sigma(s)=[1+\lambda+a\lambda]+[b\lambda]i,
\]
 where we have defined $\lambda:=\frac{1}{2}(\kappa-1)\geq0$. Further
define $\alpha:=1+\lambda+a\lambda$ and $\beta:=b\lambda$, so that
\[
\sigma(s)=\left(\alpha+\beta i\right)\sigma_{\min}.
\]

Then we can write 
\begin{align*}
\left|e^{-\frac{x^{2}}{2\sigma(s)^{2}}}\right| & =\left|\exp\left[-\left(\frac{x}{\sqrt{2}\sigma_{\min}}\right)^{2}(\alpha+\beta i)^{-2}\right]\right|\\
 & =\exp\left[-\left(\frac{x}{\sqrt{2}\sigma_{\min}}\right)^{2}\Re\left[(\alpha+\beta i)^{-2}\right]\right]\\
 & =\exp\left[\left(\frac{x}{\sqrt{2}\sigma_{\min}}\right)^{2}\frac{\beta^{2}-\alpha^{2}}{(\alpha^{2}+\beta^{2})^{2}}\right].
\end{align*}
 It will follow that 
\[
\left|e^{-\frac{x^{2}}{2\sigma(s)^{2}}}\right|\leq1
\]
 as long as we can guarantee that $\beta^{2}\leq\alpha^{2}$.

Then we claim that under the choice 
\begin{equation}
\rho=\frac{\lambda+1}{\lambda},\label{eq:rhochoice}
\end{equation}
 the desired inequality (\ref{eq:ellipsebound}) holds with $M=1$.
The preceding argument establishes that to prove this claim, it suffices
to show that $\beta^{2}\leq\alpha^{2}$ whenever $s\in E_{\rho}$.

To wit, we want to show that 
\[
(b\lambda)^{2}\leq(1+\lambda+a\lambda)^{2}
\]
 for all $a+ib\in E_{\rho}$. To see this, in turn it suffices to
show that 
\[
\vert b\vert\lambda\leq1+\lambda+a\lambda.
\]
 Now recall that any $s=a+ib\in E_{\rho}$ satisfies $\vert a\vert\leq\frac{1}{2}(\rho+\rho^{-1})$
and $\vert b\vert\leq\frac{1}{2}(\rho-\rho^{-1})$. Thus in turn it
suffices to show that 
\begin{equation}
\frac{1}{2}(\rho-\rho^{-1})\lambda\leq1+\lambda-\frac{1}{2}(\rho+\rho^{-1})\lambda.\label{eq:claimineq}
\end{equation}
 Algebraic manipulations verify that the choice (\ref{eq:rhochoice})
ensures that (\ref{eq:claimineq}) holds with equality.

Therefore (\ref{eq:ellipsebound}) holds with $M=1$ and $\rho=\frac{\lambda+1}{\lambda}$.
Since in turn $\lambda=\frac{1}{2}(\kappa-1)$, the statement of the
lemma directly follows from substitution and simplification.
\end{proof}
\vspace{2mm} \noindent \begin{center} \rule{0.5\columnwidth}{0.5pt}
\end{center} \vspace{2mm}
\begin{proof}
[Proof of Lemma \ref{lem:sigma}] 

We introduce an intermediate quantity $\mathcal{B}_{t}\mathcal{B}_{t}^{(N_{\sigma})*}$
and then expand the definitions of $\mathcal{B}_{t}$ (\ref{eq:Bt})
and $\mathcal{B}_{t}^{(N_{\sigma})}$ (\ref{eq:BtNsig}), recalling
our assumption (\ref{eq:w_assumption}) that $|w(x)|\leq1$:

\begin{align*}
 & \left|\left(\mathcal{B}_{t}\mathcal{B}_{t}^{*}-\mathcal{B}_{t}^{(N_{\sigma})}\mathcal{B}_{t}^{(N_{\sigma})*}\right)(x,y)\right|\\
 & \quad\leq\ \ \left|\left(\mathcal{B}_{t}\mathcal{B}_{t}^{*}-\mathcal{B}_{t}\mathcal{B}_{t}^{(N_{\sigma})*}\right)(x,y)\right|+\left|\left(\mathcal{B}_{t}\mathcal{B}_{t}^{(N_{\sigma})*}-\mathcal{B}_{t}^{(N_{\sigma})}\mathcal{B}_{t}^{(N_{\sigma})*}\right)(x,y)\right|\\
 & \quad\leq\ \ \left|\int_{\R^{d}}e^{-\frac{\vert x-z\vert^{2}}{2\sigma^{2}(x)\chi^{2}(t)}}e^{-\frac{\vert y-z\vert^{2}}{2\sigma^{2}(y)\chi^{2}(t)}}\,dz-\sum_{j=0}^{N_{\sigma}}P_{j}(\sigma(x))\int_{\R^{d}}e^{-\frac{\vert x-z\vert^{2}}{2\sigma_{j}^{2}\chi(t)^{2}}}e^{-\frac{\vert y-z\vert^{2}}{2\sigma^{2}(y)\chi^{2}(t)}}\,dz\right|\\
 & \quad\quad\quad+\ \ \left|\sum_{j=0}^{N_{\sigma}}P_{j}(\sigma(x))\int_{\R^{d}}e^{-\frac{\vert x-z\vert^{2}}{2\sigma_{j}^{2}\chi(t)^{2}}}e^{-\frac{\vert y-z\vert^{2}}{2\sigma^{2}(y)\chi^{2}(t)}}\,dz-\sum_{j,k=0}^{N_{\sigma}}P_{j}(\sigma(x))P_{k}(\sigma(y))\int_{\R^{d}}e^{-\frac{\vert x-z\vert^{2}}{2\sigma_{j}^{2}\chi^{2}(t)}}e^{-\frac{\vert y-z\vert^{2}}{2\sigma_{k}^{2}\chi^{2}(t)}}\,dz\right|\\
 & \quad=\ \ \left|\int_{\R^{d}}e^{-\frac{\vert y-z\vert^{2}}{2\sigma^{2}(y)\chi^{2}(t)}}\left(e^{-\frac{\vert x-z\vert^{2}}{2\sigma^{2}(x)\chi^{2}(t)}}-\sum_{j=0}^{N_{\sigma}}P_{j}(\sigma(x))e^{-\frac{\vert x-z\vert^{2}}{2\sigma_{j}^{2}\chi(t)^{2}}}\right)\,dz\right|\\
 & \quad\quad\quad+\ \ \left|\sum_{j=0}^{N_{\sigma}}P_{j}(\sigma(x))\int_{\R^{d}}e^{-\frac{\vert x-z\vert^{2}}{2\sigma_{j}^{2}\chi^{2}(t)}}\left(e^{-\frac{\vert y-z\vert^{2}}{2\sigma^{2}(y)\chi^{2}(t)}}-\sum_{k=0}^{N_{\sigma}}P_{k}(\sigma(y))e^{-\frac{\vert y-z\vert^{2}}{2\sigma_{k}^{2}\chi^{2}(t)}}\right)\,dz\right|\\
 & \quad\leq\ \ \int_{\R^{d}}e^{-\frac{\vert y-z\vert^{2}}{2\sigma^{2}(y)\chi^{2}(t)}}\left|e^{-\frac{\vert x-z\vert^{2}}{2\sigma^{2}(x)\chi^{2}(t)}}-\sum_{j=0}^{N_{\sigma}}P_{j}(\sigma(x))e^{-\frac{\vert x-z\vert^{2}}{2\sigma_{j}^{2}\chi(t)^{2}}}\right|\,dz\\
 & \quad\quad\quad+\ \ \sum_{j=0}^{N_{\sigma}}\left|P_{j}(\sigma(x))\right|\int_{\R^{d}}e^{-\frac{\vert x-z\vert^{2}}{2\sigma_{j}^{2}\chi^{2}(t)}}\left|e^{-\frac{\vert y-z\vert^{2}}{2\sigma^{2}(y)\chi^{2}(t)}}-\sum_{k=0}^{N_{\sigma}}P_{k}(\sigma(y))e^{-\frac{\vert y-z\vert^{2}}{2\sigma_{k}^{2}\chi^{2}(t)}}\right|\,dz
\end{align*}
 Now we apply Lemma \ref{lem:interp} twice to bound each of the differences
between absolute value bars in the last expression, yielding: 
\begin{equation}
\left|\left(\mathcal{B}_{t}\mathcal{B}_{t}^{*}-\mathcal{B}_{t}^{(N_{\sigma})}\mathcal{B}_{t}^{(N_{\sigma})*}\right)(x,y)\right|\leq\epsilon_{\mathrm{cheb}}\int_{\R^{d}}e^{-\frac{\vert y-z\vert^{2}}{2\sigma^{2}(y)\chi^{2}(t)}}\,dz+\epsilon_{\mathrm{cheb}}\sum_{j=0}^{N_{\sigma}}\left|P_{j}(\sigma(x))\right|\int_{\R^{d}}e^{-\frac{\vert x-z\vert^{2}}{2\sigma_{j}^{2}\chi^{2}(t)}}\,dz.\label{eq:applyinterp}
\end{equation}
 Then for $\sigma\leq\sigma_{\max}$ and $\chi\leq\chi_{\max}$, we
have 
\[
\int_{\R^{d}}e^{-\frac{\vert y-z\vert^{2}}{2\sigma^{2}\chi^{2}}}\,dz=O(\rho_{\max}^{d}).
\]
 Therefore from (\ref{eq:applyinterp}) and Definition \ref{def:lebesgue}
(of the Lebesgue constant), it follows that 
\[
\left|\left(\mathcal{B}_{t}\mathcal{B}_{t}^{*}-\mathcal{B}_{t}^{(N_{\sigma})}\mathcal{B}_{t}^{(N_{\sigma})*}\right)(x,y)\right|=O\left(\epsilon_{\mathrm{cheb}}(1+\Lambda_{N_{\sigma}})\rho_{\max}^{d}\right),
\]
 independently of $x,y$, from which the desired statement follows.
\end{proof}
\vspace{2mm} \noindent \begin{center} \rule{0.5\columnwidth}{0.5pt}
\end{center} \vspace{2mm}
\begin{proof}
[Proof of Lemma \ref{lem:stage3warmup}] For all $k=0,\ldots,N_{\sigma}$,
define 
\[
g_{k,t,x}(z)=e^{-\frac{\vert x-z\vert^{2}}{2\sigma_{k}^{2}\chi^{2}(t)}},
\]
 so that $\mathcal{G}_{k,t}$ (defined as in (\ref{eq:Gkt})) satisfies
\[
\mathcal{G}_{k,t}(x,z)=g_{k,t,x}(z).
\]
 Moreover, let $\hat{g}_{k,t,x}:=\mathcal{F}g_{k,t,x}$ denote the
Fourier transform of this function.

Define the index set $\mathbf{G}:=\{-M,\ldots,M\}^{d}$ for our Fourier
grid, and let $\omega_{\mathbf{n}}:=\mathbf{n}\Delta\omega$ denote
the Fourier grid points themselves for $\mathbf{n}\in\mathbf{G}$.

Then observe that 
\begin{align*}
(\Delta\omega)^{d}\left[\mathcal{G}_{k,t}\mathcal{F}^{*}\Pi^{*}\Pi\mathcal{F}\mathcal{G}_{l,t}\right](x,y) & \ =\ (\Delta\omega)^{d}\left[\left(\Pi\mathcal{F}\mathcal{G}_{k,t}\right)^{*}\left(\Pi\mathcal{F}\mathcal{G}_{l,t}\right)\right](x,y)\\
 & \ =\ (\Delta\omega)^{d}\sum_{\mathbf{n}\in\mathbf{G}}\overline{\hat{g}_{k,t,x}(\omega_{\mathbf{n}})}\,\hat{g}_{l,t,y}(\omega_{\mathbf{n}}),
\end{align*}
and meanwhile 
\[
\left[\mathcal{G}_{k,t}\mathcal{G}_{l,t}\right](x,y)=\left[\mathcal{G}_{k,t}\mathcal{F}^{*}\mathcal{F}\mathcal{G}_{l,t}\right](x,y)=\int\overline{\hat{g}_{k,t,x}(\omega)}\,\hat{g}_{l,t,y}(\omega)\,d\omega.
\]
 Therefore to prove the lemma, we simply need to bound the error of
the trapezoidal rule on our Fourier grid defined by $(M,\Delta\omega)$
for the integration of the function 
\[
\hat{f}_{k,l,x,y,t}(\omega):=\overline{\hat{g}_{k,t,x}(\omega)}\,\hat{g}_{l,t,y}(\omega),
\]
 independently of $t\in[t_{\min},t_{\max}]$, $k,l\in\{0,\ldots,N_{\sigma}\}$,
and $x,y\in[-1,1]^{d}$. We will write $\hat{f}=\hat{f}_{k,l,x,y,t}$
for notational simplicity. Concretely then we want to bound the integral
error 
\[
I:=\left|\int_{\R^{d}}\hat{f}(\omega)\,d\omega-(\Delta\omega)^{d}\sum_{\mathbf{n}\in\mathbf{G}}\hat{f}(\mathbf{n}\,\Delta\omega)\right|.
\]

We can compute directly the Fourier transform of the Gaussian
\[
\hat{g}_{k,t,x}(\omega)=\left(\sqrt{2\pi}\sigma_{k}\chi(t)\right)^{d}e^{-2\pi i\omega\cdot x}e^{-\left[2\pi\sigma_{k}\chi(t)\right]^{2}\,\frac{\vert\omega\vert^{2}}{2}}
\]
and hence also the analytical formula
\begin{equation}
\hat{f}(\omega)=\left(2\pi\sigma_{k}\sigma_{l}\chi^{2}(t)\right)^{d}e^{2\pi i\omega\cdot(x-y)}e^{-\frac{\left[2\pi\sigma_{k}\chi(t)\right]^{2}+\left[2\pi\sigma_{l}\chi(t)\right]^{2}}{2}\,\vert\omega\vert^{2}.}\label{eq:fhat}
\end{equation}

Then bound 
\begin{equation}
I\leq\underbrace{\left|\int_{\R^{d}}\hat{f}(\omega)\,d\omega-(\Delta\omega)^{d}\sum_{\mathbf{n}\in\mathbb{Z}^{d}}\hat{f}(\mathbf{n}\,\Delta\omega)\right|}_{=:\,I_{1}}+\underbrace{(\Delta\omega)^{d}\sum_{\mathbf{n}\in\mathbb{Z}^{d}\backslash\mathbf{G}}\vert\hat{f}(\mathbf{n}\,\Delta\omega)\vert}_{=:\,I_{2}}.\label{eq:Ibound}
\end{equation}
The first term on the right-hand side of (\ref{eq:Ibound}) measures
the error of the infinite Riemann sum approximation of $\int_{\R^{d}}\hat{f}(\omega)\,d\omega$.
The second term measures the effect of cutting off the tails of $\hat{f}$.

To bound the second term, we deduce from (\ref{eq:fhat}) that 
\begin{equation}
\vert\hat{f}(\omega)\vert=O\left(\rho_{\max}^{2d}e^{-[2\pi\rho_{\min}\vert\omega\vert]^{2}}\right),\label{eq:absfhat}
\end{equation}
 from which it follows\footnote{We can bound $I_{2}\leq C\rho_{\max}^{2d}(\Delta\omega)^{d}\sum_{\mathbf{n}\in\mathbb{Z}^{d}\backslash\mathbf{G}}e^{-[2\pi\rho_{\min}\Delta\omega]^{2}\vert\mathbf{n}\vert^{2}}$
using (\ref{eq:absfhat}), for a suitable constant $C>0$. Then we
can bound the sum over $\mathbb{Z}^{d}\backslash\mathbf{G}$ with
the sum of the sums over each of the `slab complements' $\{\mathbf{n}\,:\,\vert n_{i}\vert>M\}$
for $i=1,\ldots,d$, which are all equal, yielding: 
\begin{align*}
I_{2} & \ \leq\ C\,d\,\rho_{\max}^{2d}(\Delta\omega)^{d}\sum_{\vert n_{1}\vert>M}\sum_{n_{2},\dots,n_{d}=-\infty}^{\infty}e^{-[2\pi\rho_{\min}\Delta\omega]^{2}\vert\mathbf{n}\vert^{2}}\\
 & \ =\ C\,d\,\rho_{\max}^{2d}\left[\Delta\omega\sum_{n=-\infty}^{\infty}e^{-[2\pi\rho_{\min}\Delta\omega]^{2}n^{2}}\right]^{d-1}\left[\Delta\omega\sum_{\vert n\vert>M}e^{-[2\pi\rho_{\min}\Delta\omega]^{2}n^{2}}\right]\\
 & \ =\ O\left(\rho_{\max}^{2d}\left[\Delta\omega+\Delta\omega\sum_{n=1}^{\infty}e^{-[2\pi\rho_{\min}\Delta\omega]^{2}n^{2}}\right]^{d-1}\left[\Delta\omega\sum_{n=M+1}^{\infty}e^{-[2\pi\rho_{\min}\Delta\omega]^{2}n^{2}}\right]\right)\\
 & \ =\ O\left(\rho_{\max}^{2d}\,\left[\Delta\omega+\int_{0}^{\infty}e^{-[2\pi\rho_{\min}u]^{2}}\,du\right]^{d-1}\left[\int_{M\Delta\omega}^{\infty}e^{-[2\pi\rho_{\min}u]^{2}}\,du\right]\right)\\
 & \ =\ O\left(\rho_{\max}^{2d}\rho_{\min}^{-d}e^{-[2\pi\rho_{\min}M\Delta\omega]^{2}}\right),
\end{align*}
 assuming $\Delta\omega=O(1)$ and using the fact \cite{chang_chernoff_2011}
that $\int_{v}^{\infty}e^{-x^{2}}\,dx=O(e^{-v^{2}})$.} that the tail term can be bounded as 
\begin{equation}
I_{2}=O\left(\rho_{\max}^{d}\lambda^{d}\,e^{-[2\pi\rho_{\min}M\Delta\omega]^{2}}\right),\label{eq:gausstail}
\end{equation}
 assuming that $\Delta\omega=O(1)$.

Now to bound the first term on the right-hand side of (\ref{eq:Ibound}),
we use the Poisson summation formula, which implies that 
\begin{equation}
(\Delta\omega)^{d}\sum_{\mathbf{n}\in\mathbb{Z}^{d}}\hat{f}(\mathbf{n}\,\Delta\omega)=\sum_{\mathbf{n}\in\mathbb{Z}^{d}}f(\mathbf{n}/\Delta\omega),\label{eq:PSF}
\end{equation}
 where $f:=\mathcal{F}^{*}\hat{f}$ is the inverse Fourier transform
of $\hat{f}$, which can be calculated explicitly as 
\begin{equation}
f(z)=\left(\frac{2\pi\sigma_{k}^{2}\sigma_{l}^{2}}{\sigma_{k}^{2}+\sigma_{l}^{2}}\chi^{2}(t)\right)^{d/2}e^{-\frac{\vert z-(y-x)\vert^{2}}{2\left(\left[\sigma_{k}\chi(t)\right]^{2}+\left[\sigma_{l}\chi(t)\right]^{2}\right)}}.\label{eq:fanalytic}
\end{equation}
 Note moreover that $\int_{\R^{d}}\hat{f}(\omega)\,d\omega=f(0)$,
so from (\ref{eq:PSF}) it follows that 
\begin{equation}
I_{1}=\left|\sum_{\mathbf{n}\in\mathbb{Z}^{d}\backslash\{0\}}f(\mathbf{n}/\Delta\omega)\right|\leq\sum_{\mathbf{n}\in\mathbb{Z}^{d}\backslash\{0\}}\left|f(\mathbf{n}/\Delta\omega)\right|.\label{eq:I1bound}
\end{equation}
 Now from the expression (\ref{eq:fanalytic}) for $f$ we deduce
that 
\begin{equation}
\vert f(z)\vert=O\left(\rho_{\max}^{d}\,e^{-\frac{\vert z-(y-x)\vert^{2}}{4\rho_{\max}^{2}}}\right),\label{eq:absfbound}
\end{equation}
 using the fact\footnote{This fact in turn follows from the identity $\frac{\sigma_{k}^{2}\sigma_{l}^{2}}{\sigma_{k}^{2}+\sigma_{l}^{2}}=\frac{\sigma_{k}^{2}}{\sigma_{k}^{2}/\sigma_{l}^{2}+1}\leq\frac{\sigma_{k}^{2}}{2}$,
which holds when $\sigma_{k}\geq\sigma_{l}$, together with similar
reasoning in the opposite case.} that $\frac{\sigma_{k}^{2}\sigma_{l}^{2}}{\sigma_{k}^{2}+\sigma_{l}^{2}}\leq\frac{\sigma_{\max}^{2}}{2}$. 

Then recall that we assume $x,y\in[-1,1]^{d}$, so in turn it must
be the case that $y-x\in[-2,2]^{d}$. Therefore from (\ref{eq:absfbound})
and the reverse triangle inequality we have that 
\begin{equation}
\vert f(z)\vert=O\left(\rho_{\max}^{d}e^{-\frac{\sum_{i=1}^{d}[\vert z_{i}\vert-2]_{+}^{2}}{4\rho_{\max}^{2}}}\right),\label{eq:absfbound2}
\end{equation}
 from which it follows\footnote{From (\ref{eq:I1bound}) and (\ref{eq:absfbound2}) we can bound $I_{1}\leq C\rho_{\max}^{d}\sum_{\mathbf{n}\in\mathbb{Z}^{d}\backslash\{0\}}e^{-\sum_{i=1}^{d}\left[\frac{n_{i}-2\Delta\omega}{2\rho_{\max}\Delta\omega}\right]_{+}^{2}}$,
for a suitable constant $C>0$. Then again we can bound the sum over
$\mathbb{Z}^{d}\backslash\{0\}$ with the sum of the sums over each
of the `slab complements' $\{\mathbf{n}\,:\,n_{i}\neq0\}$ for $i=1,\ldots,d$,
which are all equal, yielding: 
\begin{align*}
I_{1} & \ \leq\ C\,d\,\rho_{\max}^{d}\left[\sum_{n\in\mathbb{Z}}e^{-\left[\frac{n-2\Delta\omega}{2\rho_{\max}\Delta\omega}\right]_{+}^{2}}\right]^{d-1}\left[\sum_{n\in\mathbb{Z}\backslash\{0\}}e^{-\left[\frac{n-2\Delta\omega}{2\rho_{\max}\Delta\omega}\right]_{+}^{2}}\right]\\
 & \ =\ O\left(\rho_{\max}^{d}\left[1+\sum_{n=1}^{\infty}e^{-\left[\frac{n-2\Delta\omega}{2\rho_{\max}\Delta\omega}\right]_{+}^{2}}\right]^{d-1}\left[\sum_{n=1}^{\infty}e^{-\left[\frac{n-2\Delta\omega}{2\rho_{\max}\Delta\omega}\right]_{+}^{2}}\right]\right).
\end{align*}
 Now if $\Delta\omega\leq\frac{1}{8}$, then 
\begin{align*}
\sum_{n=1}^{\infty}e^{-\left[\frac{n-2\Delta\omega}{2\rho_{\max}\Delta\omega}\right]_{+}^{2}} & \leq\sum_{n=1}^{\infty}e^{-\left[\frac{n-1/4}{2\rho_{\max}\Delta\omega}\right]^{2}}\\
 & \leq\sum_{n=4}^{\infty}e^{-\left[\frac{n/4-1/4}{2\rho_{\max}\Delta\omega}\right]^{2}}\leq4\int_{3/4}^{\infty}e^{-\left[\frac{u-1/4}{2\rho_{\max}\Delta\omega}\right]^{2}}\,du=4\int_{1/2}^{\infty}e^{-\left[\frac{u}{2\rho_{\max}\Delta\omega}\right]^{2}}\,du=O\left(e^{-\left[\frac{1}{4\rho_{\max}\Delta\omega}\right]^{2}}\right),
\end{align*}
 again using the fact that $\int_{v}^{\infty}e^{-x^{2}}\,dx=O(e^{-v^{2}})$.

Then
\begin{align*}
I_{1} & \ =\ O\left(\rho_{\max}^{d}\left[1+\sum_{n=1}^{\infty}e^{-\left[\frac{n-2\Delta\omega}{2\rho_{\max}\Delta\omega}\right]_{+}^{2}}\right]^{d-1}\left[\sum_{n=1}^{\infty}e^{-\left[\frac{n-2\Delta\omega}{2\rho_{\max}\Delta\omega}\right]_{+}^{2}}\right]\right)\\
 & \ =\ O\left(\rho_{\max}^{d}\left[1+e^{-\left[\frac{1}{4\rho_{\max}\Delta\omega}\right]^{2}}\right]^{d-1}\left[e^{-\left[\frac{1}{4\rho_{\max}\Delta\omega}\right]^{2}}\right]\right)\\
 & \ =\ O\left(\rho_{\max}^{d}2^{d-1}e^{-\left[\frac{1}{4\rho_{\max}\Delta\omega}\right]^{2}}\right)\\
 & \ =\ O\left(\rho_{\max}^{d}e^{-\left[\frac{1}{4\rho_{\max}\Delta\omega}\right]^{2}}\right)
\end{align*}
since we treat $d$ as constant.} that $I_{1}$ can be bounded as 
\[
I_{1}=O\left(\rho_{\max}^{d}e^{-\left[\frac{1}{4\rho_{\max}\Delta\omega}\right]^{2}}\right),
\]
 provided that $\Delta\omega\leq1/8$. Combining with (\ref{eq:gausstail})
in (\ref{eq:Ibound}) completes the proof.
\end{proof}
\vspace{2mm} \noindent \begin{center} \rule{0.5\columnwidth}{0.5pt}
\end{center} \vspace{2mm}
\begin{proof}
[Proof of Lemma \ref{lem:stage3}] We employ definition \ref{eq:BWG}
to write

\begin{align*}
A:= & \trip(\Delta\omega)^{d}\,\mathcal{B}_{t}^{(N_{\sigma})}\mathcal{F}^{*}\Pi^{*}\Pi\mathcal{F}\mathcal{B}_{t}^{(N_{\sigma})*}-\mathcal{B}_{t}^{(N_{\sigma})}\mathcal{B}_{t}^{(N_{\sigma})*}\trip\\
 & \quad=\ \lrtrip{(\Delta\omega)^{d}\,\left(\sum_{k=0}^{N_{\sigma}}\mathcal{W}_{k}\mathcal{G}_{k,t}\right)\mathcal{F}^{*}\Pi^{*}\Pi\mathcal{F}\left(\sum_{k=0}^{N_{\sigma}}\mathcal{G}_{k,t}\mathcal{W}_{k}\right)-\left(\sum_{k=0}^{N_{\sigma}}\mathcal{W}_{k}\mathcal{G}_{k,t}\right)\left(\sum_{k=0}^{N_{\sigma}}\mathcal{G}_{k,t}\mathcal{W}_{k}\right)}\\
 & \quad\leq\ \sum_{k,l=0}^{N_{\sigma}}\lrtrip{\mathcal{W}_{k}\left[(\Delta\omega)^{d}\,\mathcal{G}_{k,t}\mathcal{F}^{*}\Pi^{*}\Pi\mathcal{F}\mathcal{G}_{l,t}-\mathcal{G}_{k,t}\mathcal{G}_{l,t}\right]\mathcal{W}_{l}}.
\end{align*}

We want to continue our calculations to derive a suitable bound on
$A$.

Now recall from \ref{eq:Wk} that $\mathcal{W}_{k}$ as a pointwise
multiplier by the function 
\begin{equation}
w_{k}(x):=w(x)\left[2\pi\,\sigma^{2}(x)\right]^{-d/2}P_{k}(\sigma(x)).\label{eq:wk}
\end{equation}
 Therefore it follows form the preceding inequality that 
\[
A\leq\sum_{k,l=0}^{N_{\sigma}}\Vert w_{k}\Vert_{\infty}\Vert w_{l}\Vert_{\infty}\lrtrip{(\Delta\omega)^{d}\,\mathcal{G}_{k,t}\mathcal{F}^{*}\Pi^{*}\Pi\mathcal{F}\mathcal{G}_{l,t}-\mathcal{G}_{k,t}\mathcal{G}_{l,t}},
\]
 where $\Vert\,\cdot\,\Vert_{\infty}$ denotes the $L^{\infty}$ norm
on $[-1,1]^{d}$. Then we can apply Lemma \ref{lem:stage3warmup}
to obtain 
\[
A=O\left(\rho_{\max}^{d}\epsilon_{\mathcal{F}}\left[\sum_{k=0}^{N_{\sigma}}\Vert w_{k}\Vert_{\infty}\right]^{2}\right).
\]

It remains to bound $\Vert w_{k}\Vert_{\infty}$. Recall our assumption
\ref{eq:w_assumption} ensuring that $\Vert w\Vert_{\infty}\leq1$,
and moreover recall that the Chebyshev interpolating functions $P_{k}$
\cite[Chapter 5]{boyd_chebyshev_2001} are bounded in magnitude by
$1$ on their domain $[\sigma_{\min},\sigma_{\max}]$. Therefore from
\ref{eq:wk} it follows that $\Vert w_{k}\Vert_{\infty}=O(\sigma_{\min}^{-d})$,
and consequently 
\[
A=O\left(N_{\sigma}^{2}\sigma_{\min}^{-2d}\rho_{\max}^{d}\epsilon_{\mathcal{F}}\right).
\]
 Since $\rho_{\max}/\sigma_{\min}^{2}=\kappa\frac{\chi_{\max}}{\sigma_{\min}}$,
we have completed the proof.
\end{proof}

\section{Proof of Theorem \ref{thm:main} \label{app:main}}
\begin{proof}
[Proof of Theorem \ref{thm:main}] We will focus first on the general
Mat\'ern case.

By the triangle inequality, we bound 
\begin{align*}
\trip\mathcal{K}-\tilde{\mathcal{K}}\trip & \leq\lrtrip{\mathcal{K}-\sum_{j=0}^{N_{t}}\mathcal{B}_{t_{j}}\mathcal{B}_{t_{j}}^{*}\,v_{j}}\\
 & \quad+\ \lrtrip{\sum_{j=0}^{N_{t}}\mathcal{B}_{t_{j}}\mathcal{B}_{t_{j}}^{*}\,v_{j}-\sum_{j=0}^{N_{t}}\mathcal{B}_{t_{j}}^{(N_{\sigma})}\mathcal{B}_{t_{j}}^{(N_{\sigma})*}\,v_{j}}\\
 & \quad+\ \lrtrip{\sum_{j=0}^{N_{t}}\mathcal{B}_{t_{j}}^{(N_{\sigma})}\mathcal{B}_{t_{j}}^{(N_{\sigma})*}\,v_{j}-\underbrace{(\Delta\omega)^{d}\sum_{j=0}^{N_{t}}\mathcal{B}_{t_{j}}^{(N_{\sigma})}\mathcal{F}^{*}\Pi^{*}\Pi\mathcal{F}\mathcal{B}_{t_{j}}^{(N_{\sigma})*}\,v_{j}}}.
\end{align*}
 The bound is justified because the underbraced expression is precisely
$\tilde{\mathcal{K}}$, following (\ref{eq:Ktilde}). Let the three
terms on the right-hand side be denoted $T_{1}$, $T_{2}$, and $T_{3}$,
respectively. 

First we cite Lemma \ref{lem:dt}, which establishes that
\begin{equation}
T_{1}=O\left(\sigma_{\min}^{-d}\epsilon_{\mathrm{trap}}\right).\label{eq:T1}
\end{equation}

Then we bound, using \ref{lem:sigma}: 
\[
T_{2}\leq\sum_{j=0}^{N_{t}}\vert v_{j}\vert\lrtrip{\mathcal{B}_{t_{j}}\mathcal{B}_{t_{j}}^{*}-\mathcal{B}_{t_{j}}^{(N_{\sigma})}\mathcal{B}_{t_{j}}^{(N_{\sigma})*}}=O\left(\Lambda_{N_{\sigma}}\,\rho_{\max}^{d}\,\epsilon_{\mathrm{cheb}}\sum_{j=0}^{N_{t}}\vert v_{j}\vert\right).
\]
 Now $\vert v_{j}\vert\leq v(t_{j})\,\Delta t$, and moreover $v(t)=\chi^{-d}(t)u(t)=O\left(e^{(\nu-\frac{d}{2})t-e^{t}}\right)$.
Thus $\nu$ is integrable, and $\sum_{j=0}^{N_{t}}\vert v_{j}\vert=O(1)$.
It follows that 
\begin{equation}
T_{2}=O\left(\Lambda_{N_{\sigma}}\,\rho_{\max}^{d}\,\epsilon_{\mathrm{cheb}}\right).\label{eq:T2}
\end{equation}

Finally, using Lemma \ref{lem:stage3}, we bound: 
\begin{align*}
T_{3} & \ \leq\ \sum_{j=0}^{N_{t}}\vert v_{j}\vert\lrtrip{\mathcal{B}_{t_{j}}^{(N_{\sigma})}\mathcal{B}_{t_{j}}^{(N_{\sigma})*}-(\Delta\omega)^{d}\sum_{j=0}^{N_{t}}\mathcal{B}_{t_{j}}^{(N_{\sigma})}\mathcal{F}^{*}\Pi^{*}\Pi\mathcal{F}\mathcal{B}_{t_{j}}^{(N_{\sigma})*}}\\
 & \ =\ O\left(N_{\sigma}^{2}\,\kappa^{d}\left[\frac{\chi_{\max}}{\sigma_{\min}}\right]^{d}\epsilon_{\mathcal{F}}\sum_{j=0}^{N_{t}}\vert v_{j}\vert\right).
\end{align*}
 Again using the fact that $\sum_{j=0}^{N_{t}}\vert v_{j}\vert=O(1)$,
we conclude that 
\begin{equation}
T_{3}=O\left(N_{\sigma}^{2}\,\kappa^{d}\left[\frac{\chi_{\max}}{\sigma_{\min}}\right]^{d}\epsilon_{\mathcal{F}}\right).\label{eq:T3}
\end{equation}
 Combining (\ref{eq:T1}), (\ref{eq:T2}), and (\ref{eq:T3}), we
conclude that 
\[
\trip\mathcal{K}-\tilde{\mathcal{K}}\trip=O\left(\sigma_{\min}^{-d}\epsilon_{\mathrm{trap}}+\Lambda_{N_{\sigma}}\,\rho_{\max}^{d}\,\epsilon_{\mathrm{cheb}}+N_{\sigma}^{2}\,\kappa^{d}\left[\frac{\chi_{\max}}{\sigma_{\min}}\right]^{d}\epsilon_{\mathcal{F}}\right),
\]
 as was to be shown.

The logic in the case of the squared-exponential kernel is exactly
the same, except that there is no integration over $t$ and so the
first error term ($T_{1}$) is discarded.
\end{proof}

\end{document}